\documentclass{amsart}

\usepackage[margin=1.3in]{geometry}

\usepackage{graphicx}
\usepackage{amsmath}
\usepackage{amssymb}
\usepackage{amsthm}
\usepackage[all]{xy}
\usepackage{hyperref}

\newcommand{\F}{\mathbb F}
\newcommand{\Q}{\mathbb Q}
\newcommand{\Z}{\mathbb Z}
\newcommand{\R}{\mathbb R}
\newcommand{\CC}{\mathbb C}

\newcommand{\fm}{\mathfrak m}
\newcommand{\fp}{\mathfrak p}
\newcommand{\fP}{\mathfrak P}
\newcommand{\co}{\mathcal{O}}
\newcommand{\fq}{\mathfrak q}

\newcommand{\Frob}{\mathrm{Frob}}
\newcommand{\Gal}{\mathrm{Gal}}

\newcommand{\Tr}{\mathrm{Tr}}
\newcommand{\bfN}{\mathbf{N}}
\newcommand{\Cl}{\mathrm{Cl}}
\newcommand{\cl}{\mathrm{cl}}

\newcommand{\Ker}{\mathrm{Ker}}

\newcommand{\RamHyp}{\mathbf{RamHyp}}
\newcommand{\ab}{\mathrm{ab}}
\newcommand{\ur}{\mathrm{un}}
\newcommand{\sgn}{\mathrm{sgn}}

\newcommand{\leg}[2]{\left(\frac{#1}{#2}\right)}
\newcommand{\qleg}[2]{\left[\frac{#1}{#2}\right]}
\newcommand{\hilbert}[3]{\Bigl(\dfrac{#1,#2}{#3}\Bigr)}
\newcommand{\pair}[1]{\langle{{#1}}\rangle}
\newcommand{\fct}[4]{\begin{split}#1&\longrightarrow#2\\#3&\longmapsto#4\end{split}}
\newcommand{\wt}[1]{\widetilde{#1}}

\allowdisplaybreaks
\numberwithin{equation}{section}

\newtheorem{thm}{Theorem}[section]

\newtheorem{prop}[thm]{Proposition}
\newtheorem{rmk}[thm]{Remark}

\newtheorem{lem}[thm]{Lemma}

\begin{document}
	\title{$\ell$-Class groups of fields in  Kummer towers}
	\author{Jianing Li}
	\address{Wu Wen-Tsun Key Laboratory of Mathematics,  School of Mathematical Sciences, University of Science and Technology of China, Hefei, Anhui 230026, China}
	\email{lijn@ustc.edu.cn}
	
	\author{Yi Ouyang}
	\address{Wu Wen-Tsun Key Laboratory of Mathematics,  School of Mathematical Sciences, University of Science and Technology of China, Hefei, Anhui 230026, China}
	
	\email{yiouyang@ustc.edu.cn}
	
	\author{Yue Xu}
	\address{Wu Wen-Tsun Key Laboratory of Mathematics,  School of Mathematical Sciences, University of Science and Technology of China, Hefei, Anhui 230026, China}
	\email{wasx250@mail.ustc.edu.cn}
	
	\author{Shenxing Zhang}
	\address{Wu Wen-Tsun Key Laboratory of Mathematics,  School of Mathematical Sciences, University of Science and Technology of China, Hefei, Anhui 230026, China}
	\email{zsxqq@mail.ustc.edu.cn}
	
	
	\maketitle
	
	\begin{abstract}
		Let $\ell$ and $p$ be prime numbers and  $K_{n,m}=\Q(p^{\frac{1}{\ell^n}},\zeta_{2\ell^{m}})$.  We study the $\ell$-class group of $K_{n,m}$ in this paper. When $\ell=2$,
		we determine the structure of the $2$-class group of $K_{n,m}$ for all $(n,m)\in \Z_{\geq 0}^2$ in the case $p=2$ or  $p\equiv 3, 5\bmod{8}$, and for $(n,m)=(n,0)$, $(n,1)$ or $(1,m)$ in the case $p\equiv 7\bmod{16}$, generalizing the results of Parry about the $2$-divisibility of the class number of $K_{2,0}$.  We also obtain results about the $\ell$-class group of $K_{n,m}$ when $\ell$ is odd and in particular $\ell=3$.
		The main tools we use are class field theory,  including Chevalley's ambiguous class number formula and its generalization by Gras, and a stationary result about  the $\ell$-class groups in the $2$-dimensional Kummer tower $\{K_{n,m}\}$.

	\end{abstract}
	
	\section{Introduction}
	In this paper we let $\ell$ and $p$  be  prime numbers. For $n$ and $m$ non-negative integers, let $K_{n,m}=\Q(p^{\frac{1}{\ell^n}},\zeta_{2\ell^{m}})$.
	Let $A_{n,m}$ and $h_{n,m}$ be the $\ell$-part of the class group and the class number of $K_{n,m}$. The aim of this paper is to study the $\ell$-class groups of $K_{n,m}$ when $n$ and $m$ vary.
	
	First let us assume $\ell=2$. It is  well-known that the class number $h_{1,0}$ of $\Q(\sqrt{p})$  is odd by the genus theory of Gauss.
	In 1886, Weber \cite{Web86} proved that the class number $h_{0,m}$ of $\Q(\zeta_{2^{m+1}})$ is odd for any $m\geq 0$. In 1980, by a more careful application of genus theory for quartic fields, Parry \cite{Par80}
	showed that $A_{2,0}$ is cyclic and
	\begin{enumerate}
		\item[(i)] If $p=2$ or $p\equiv 3,5 \bmod 8$,  then $2\nmid h_{2,0}$.
		\item[(ii)] If $p\equiv 7\bmod 16$, then $2\parallel h_{2,0}$.
		\item[(iii)] If $p\equiv 15 \bmod 16$, then $ 2\mid h_{2,0}$.
		\item[(iv)] If $p\equiv 1 \bmod 8$, then $2\mid h_{2,0}$. Moreover, if $2$ is not a fourth power modulo $p$,   then $2\parallel h_{2,0}$.
	\end{enumerate}
	For $p\equiv 9 \bmod 16$, Lemmermeyer  showed that $2\parallel h_{2,0}$,  see \cite{Mon10}. For $p\equiv 15\bmod 16$, one can show that  $4\mid h_{2,0}$ using genus theory (unpublished manuscripts by the authors  and Lemmermeyer respectively).
	
	Our first result of this paper is
	\begin{thm}\label{thm:main1}
		Let $p$ be a prime number, $K_{n,m}=\Q(p^{\frac{1}{2^n}},\zeta_{2^{m+1}})$.
		Let $A_{n,m}$ be the $2$-part of the class group and $h_{n,m}$ the class number of $K_{n,m}$.
		
		$(1)$ If $p=2$ or $p\equiv 3 \bmod 8$,  then $h_{n,m}$ is odd for   $n,m\ge 0$.
		
		$(2)$  If $p\equiv 5 \bmod 8$, then $h_{n,0}$ and $h_{1,m}$ are odd for  $n,m\ge 0$ and $2\parallel h_{n,m}$ for   $n\geq 2$ and $m\geq 1$.
		
		$(3)$ If $p\equiv 7 \bmod 16$, then $A_{n,0}\cong\Z/{2\Z}, A_{n,1}\cong\Z/{2\Z}\times \Z/{2\Z}$ for  $n\ge 2$, and $A_{1,m}\cong \Z/{2^{m-1}\Z}$ for  $m\geq 1$.
	\end{thm}
	

	Let $p\equiv 3\bmod 8$ and $\epsilon=a+b\sqrt{p}$ be the  fundamental unit of $\Q(\sqrt{p})$. Parry~\cite{Par80} and Zhang-Yue~\cite{ZY14} showed that $a \equiv -1 \bmod {p}$ and $v_2(a)=1$.
	Applying Theorem~\ref{thm:main1}, we obtain the following  analogue of their results.

	\begin{thm} \label{thm: units}  Assume $p\equiv 7\bmod 16$. Let $\epsilon$ be the fundamental unit of $\Q(\sqrt{p})$.

		$(1)$   There exists a totally positive unit $\eta$ of  $\Q(\sqrt[4]{p})$ such that $\bfN(\eta)=\epsilon$ and  the group of units $\co^\times_{\Q(\sqrt[4]{p})}=\langle \eta, \epsilon, -1\rangle$.
		
		$(2)$  For any unit $\eta' \in \bfN^{-1}(\epsilon)$ in $\Q(\sqrt[4]{p})$, one has $ v_{\fq}(\Tr_{\Q(\sqrt[4]{p})/\Q(\sqrt{p})}(\eta'))=3$ and $\eta' \equiv -\sgn(\eta') \bmod \sqrt[4]{p}$,  where $\fq$ is the unique prime of $\Q(\sqrt{p})$ above $2$ and $\sgn$ is the signature function.
		
	\end{thm}

	\begin{rmk}\label{rmk: untis} 
		$(1)$ We may call the unit $\eta$  the relative fundamental unit of $\Q(\sqrt[4]{p})$.
		The first part of this theorem is due to Parry, see \cite[Theorem 3]{Par80}. We include a proof here for completeness. 
		
		$(2)$  For $\eta' \in \co^\times_{\Q(\sqrt[4]{p})}$ such that  $\bfN(\eta')=\epsilon$,  we know $\eta'$ is either totally positive or totally negative since  $\epsilon$ is totally positive.  Therefore the sign of $\eta'$ is well-defined.
	\end{rmk}
	
	Now assume $\ell$ is odd. Recall that $\ell$ is regular if $\ell\nmid h_{0,1}$, the class number of $\Q(\zeta_\ell)$. We have the following result:
	\begin{thm}\label{thm: odd} 
		Assume  $\ell$ is an odd  regular prime, and  $p$ is either $\ell$ or a prime  generating the group $(\Z/{\ell^2\Z})^\times$. Then $\ell\nmid h_{n,m}$, the class number of $K_{n,m}=\Q(p^{\frac{1}{\ell^n}},\zeta_{\ell^{m}})$ for any $n,m\geq 0$. 
	\end{thm}

	For the particular case $\ell=3$,  the following results about the $3$-class groups of $\Q(\sqrt[3]{p})$ and $\Q(\sqrt[3]{p}, \zeta_3)$ were obtained by several authors:
	\begin{enumerate}
		\item[(i)] (\cite{Hon71}) If $p=3$ or $p\equiv 2 \bmod 3$,  then $3\nmid h_{1,1}$ and $3\nmid h_{1,0}$. 
		\item[(ii)] (\cite{Ger76} ) If $p\equiv 1\bmod 3$,  then $\mathrm{rank}_3 A_{1,0}=1$  and $\mathrm{rank}_3 A_{1,1}=1 \text{ or } 2$. 
		\item[(iii)] (\cite{Aouissi}) If $p\equiv 4,7\bmod 9$, then $A_{1,0}\cong \Z/{3\Z}$ and 
		\[ A_{1,1}\cong \begin{cases}
		\Z/{3\Z}  & \text{ if } \leg{3}{p}_3\neq 1,\\
		(\Z/{3\Z})^2  & \text{ if } \leg{3}{p}_3= 1.\\
		\end{cases}  \]
		\item[(iv)] (\cite{C-E05}, \cite{Ger05}) If $p\equiv 1\bmod 9$, then $\mathrm{rank}_3 A_{1,1}=1$ if and only if $9\mid h_{1,0}$.
	\end{enumerate}
	We refer to \cite{Ger05} and  \cite{Aouissi} for more details. However, $h_{n,m}$ and $A_{n,m}$ for general $n$ and $m$ was rarely studied in the literature as far as we know.
	We have the following result in this case: 
	\begin{thm}\label{thm:main2}
		Let $p$ be a prime number. Let $A_{n,m}$ be the $3$-part of the class group and $h_{n,m}$ the class number of $K_{n,m}=\Q(p^{\frac{1}{3^n}},\zeta_{3^{m}})$. 
		
		$(1)$ If $p=3$ or $p\equiv 2,5\bmod 9$,  then $3\nmid h_{n,m}$ for   $n,m\ge 0$.
		
		$(2)$  If $p\equiv 4,7 \bmod 9$ and the cubic residue symbol $\leg{3}{p}_3\neq 1$, then $A_{n,m}\cong \Z/{3\Z}$ for $n\geq 1$, $m\geq 0$.  
	\end{thm}
	
	\begin{rmk}
		A. Lei~\cite{Lei17} obtained the growth formula of class numbers in $\Z^{d-1}_{\ell}\rtimes \Z_\ell$-extensions for an odd prime $\ell$. Under the conditions in Theorem~\ref{thm: odd} or \ref{thm:main2}, the Kummer tower $K_{\infty,\infty}/K_{0,1}$ satisfies the conditions in Lei's paper. Then by \cite[Corollary 3.4]{Lei17}, one has  for each $m$, there exist integers $\mu_m$ and $\lambda_m$ such that  \[ v_\ell(h_{n,m})=\mu_m\ell^n+ \lambda_m n + O(1) \text{ for } n\gg 0 .\]
		Theorem~\ref{thm: odd} and \ref{thm:main2} thus imply that the invariants $\mu_m=\lambda_m=0$ for all $m$.
	\end{rmk}
	
	To prove our results,  we need to use class field theory,  including Chevalley's ambiguous class number formula and its generalization by Gras. The most technical part of our paper is  a stationary result of  $\ell$-class groups in a cyclic $\Z/\ell^2\Z$-extension under certain conditions, and its application to the study of $\ell$-class groups in the $2$-dimensional Kummer tower $\{K_{n,m}\}$. We emphasize that the stationary result could be used to other situations. Due to the computational nature of our results, we impose conditions to simplify computation. It would be of interest to study other cases, for example,  replacing $p$ by some positive integer with $2$ or more prime factors.

	The organization of this paper is as follows. In \S2 we introduce notations and conventions for the paper, and present basic properties of the Hilbert symbols and Gras' formula on genus theory.
	In \S3, we prove our stationary result  on $\ell$-class groups in certain cyclic $\ell$-extensions by using argument from Iwasawa theory, and then prove a stationary result about  the $\ell$-class groups of $K_{n,m}$.  \S4 is devoted to the proof of results for the easier case that $\ell$ is odd and \S5 for the more complicated case $\ell=2$.

	\section{Preliminary}
	
	\subsection{Notations and Conventions} The numbers $\ell$ and $p$ are always   prime numbers. The $\ell$-Sylow subgroup of a finite abelian group $M$ is denoted by $M(\ell)$. $\zeta_n$ is a primitive $n$-th root of unity and $\mu_n$ is the group of $n$-th roots of unity.
	
	For a number field $K$, we denote by  $\Cl_K$, $h_K$,  $\co_k$, $E_K$ and $\cl$   the class group, the class number, the ring of integers, the unit group of the ring of integers  and the ideal class map of $K$ respectively.  For $w$ a place of $K$, $K_w$ is the completion of $K$ by  $w$. For $\fp$ a prime of $K$, $v_{\fp}$ is the additive valuation associated to $\fp$.

	For an extension  $K/F$  of number fields,  $v$  a  place of $F$ and $w$ a place of $K$ above $v$, let $e_{w/v}=e(w/v,K/F)$ be the ramification index in $K/F$ if $v$ is finite and $e_{w/v}=[K_w:F_v]$ if $v$ is infinite.
	We say that $w/v$ is ramified if $e_{w/v}>1$. $w/v$ is totally ramified if $e_{w/v}=[K:F]$, in this case $w$ is the only place above $v$ and we can also say that $v$ is totally ramified in $K/F$.
	Note that for $v$ infinite, $w/v$ is ramified if and only if $w$ is complex and $v$ is real, and in this case $e_{w/v}=2$. Hence an infinite place $v$ is totally ramified if and only if $K/F$ is quadratic, $F_v=\R$  and $K_w=\CC$.
	When $K/F$ is Galois, then $e_{w/v}$ is independent of $w$ and we denote it by $e_v$.

	Denote by $\bfN_{K/F}$ the norm map from $K$ to $F$, and the induced norm map  from $\Cl_K$ to $\Cl_F$. If the extension is clear, we use $\bfN$ instead of $\bfN_{K/F}$.
	
	When $K=K_{n,m}=\Q(p^{\frac{1}{\ell^n}},\zeta_{2^{m+1}})$, we write  $\Cl_{n,m}=\Cl_K$, $h_{n,m}=h_K$, $\co_{n,m}=\co_K$ and $E_{n,m}=E_K$ for simplicity. The group $A_{n,m}$ is the $\ell$-Sylow subgroup of $\Cl_{n,m}$.
	\subsection{Hilbert symbol}
	Let $n\geq 2$ be an integer.   Let $k$ be a finite extension of $\Q_p$ containing $\mu_n$. Let $\phi_k$ be the local reciprocity map $\phi_k: k^\times \longrightarrow \Gal(k^{\ab}/k)$. Given $a,b\in k^\times$, the $n$-th Hilbert symbol is defined by
	\[ \hilbert{a}{b}{k}_n=\frac{\phi_k(a)(\sqrt[n]{b})}{\sqrt[n]{b}} \in \mu_n\subset k. \]
	The following results about Hilbert symbol can be found in standard textbooks in number theory, for example \cite[Chapters IV and V]{Neu13}.
	\begin{prop}\label{prop: hil}
		Let $a,b\in k^\times$.

		$(1)$ $\hilbert{a}{b}{k}_n=1 \Leftrightarrow $ $a$ is a norm from the extension $k(\sqrt[n]{b})/k$;
		
		$(2)$ $\hilbert{aa'}{b}{k}_n=\hilbert{a}{b}{k}_n \hilbert{a'}{b}{k}_n$ and $\hilbert{a}{bb'}{k}_n=\hilbert{a}{b}{k}_n
		\hilbert{a}{b'}{k}_n$;
		
		$(3)$ $\hilbert{a}{b}{k}_n=\hilbert{b}{a}{k}^{-1}_n$;
		
		$(4)$ $\hilbert{a}{1-a}{k}_n=1$ and $\hilbert{a}{-a}{k}_n=1$;
		
		$(5)$ Let $\varpi$ be a uniformizer of $k$. Let $q=|{\co_k}/{(\varpi)}|$ be the cardinality of the residue field of $k$. If $p\nmid n$, then  $\hilbert{\varpi}{u}{k}_n=\omega(u)^{\frac{q-1}{n}}$ where $\omega: \co^\times_{k} \rightarrow \zeta_{q-1}$ is the unique map such that $u\equiv \omega(u) \bmod \varpi$ for $u\in \co^\times_k$.
		
		$(6)$ Let $M/k$ be a finite extension. For $a\in M^\times,b\in k^\times$, one has the following norm-compatible property
		\[ \hilbert{a}{b}{M}_n=\hilbert{\bfN_{M/k}(a)}{b}{k}_n. \]
	\end{prop}

	When $k=\R$, $\mu_n\subset \R$ if and only if $n=1$ or $2$.   For $a,b\in k^\times$ define
	\[ \hilbert{a}{b}{k}_2=\begin{cases}
	-1, & \text{ if } a<0\ \text{and}\ b<0;\\
	1, & \text{ otherwise}.\\
	\end{cases} \]

	When $k=\CC$, define $\hilbert{a}{b}{k}_n=1$ for any $a,b\in k^\times$.

	The following is the product formula of Hilbert symbols, see  \cite[Chapter VI, Theorem 8.1]{Neu13}.
	
	\begin{prop}Let $K$ be a number field such that $\mu_n\subset K$. For any place $v$ of $K$, set $\hilbert{a}{b}{v}_n:= \iota_v^{-1}\left(\hilbert{a}{b}{K_v}_n\right)$ where $\iota_v$ is the canonical embedding of $K\rightarrow K_v$. Then for
		$a,b \in K^\times$, one has
		\[ \prod_{v}\hilbert{a}{b}{v}_n=1, \]
		where $v$  runs over all places of  $K$.
	\end{prop}
	\subsection{Three useful Lemmas}
	
	\begin{lem}\label{nakayama}
		Suppose  $K/F$ is a cyclic $\ell$-extension with Galois group $G$ and $C$ is a $G$-submodule of $\Cl_K$.  Then  $\ell \nmid |(\Cl_{K}/C)^G|$ implies that  $\Cl_K(\ell)=C(\ell)$. In particular, $\ell \nmid |\Cl_{K}^G|$ implies that $\ell \nmid h_K$.
	\end{lem}
	
	\begin{proof}
		
		Consider the  action of  $G$ on $(\Cl_K/C)(\ell)$. The cardinality of the orbit of  $c \in (\Cl_K/C)(\ell)\setminus (\Cl_K/C)(\ell)^G$ is a multiple of $\ell$.  Thus  $|(\Cl_K/C)(\ell)|\equiv |(\Cl_K/C)(\ell)^G|  \bmod \ell$.   Hence  $\ell \nmid |(\Cl_{K}/C)^G|$ implies $(\Cl_K/C)(\ell)=0$ and then $\Cl_K(\ell)=C(\ell)$ by  the  exact sequence $0 \rightarrow C(\ell) \rightarrow \Cl_K(\ell)\rightarrow (\Cl_K/C)(\ell)$. \end{proof}

	\begin{lem}\label{lem: ram}
		Let $K_n/K_0$ be a cyclic extension of number fields of degree $\ell^n$.  Let $K_i$ be the unique intermediate field  such that $[K_i:K_0]=\ell^i$ for $0\leq i \leq n$.     If a prime ideal $\fp$ of $K_0$ is ramified in $K_1/K_0$, then $\fp$   is totally ramified in $K_n/K_0$.
	\end{lem}
	
	\begin{proof}
		Let $I_\fp$ be the inertia group  of $\fp$. Then $K^{I_\fp}_n=K_i$ for some $i$ and  $K^{I_\fp}_n/K$ is unramified at $\fp$.   Since $K_1/K_0$ is ramified at $\fp$, we must have $K^{I_\fp}_n=K_0$. In other words, $\fp$ is totally ramified. \end{proof}

	\begin{lem}\label{prop: norm_surj}
		Suppose the number field extension $M/K$  contains no unramified abelian sub-extension other than $K$. Then the norm map $\Cl_M \rightarrow \Cl_K$ is surjective. In particular, $h_K\mid h_M$.
	\end{lem}
	\begin{proof} This is \cite[Theorem 10.1]{Was97}.
	\end{proof}
	
	\subsection{Gras' formula on  class groups in cyclic extensions}
	
	\begin{thm}[Gras]\label{Gras's thm}
		Let $K/F$ be a cyclic extension of number fields with Galois group $G$.
		Let $C$ be a $G$-submodule of $\Cl_K$. Let $D$ be a subgroup of fractional ideals of $K$  such that $\cl(D)=C$. Denote by $\Lambda_D=\{x\in F^\times\mid (x)\co_F \in \bfN D\}$. Then
		\begin{equation} \label{eq:Gras} |(\Cl_K/C)^G|=\frac{|\Cl_F|}{|\bfN C|}\cdot \frac{\prod_{v}e_v}{[K:F]}\cdot \frac{1}{[\Lambda_D:\Lambda_D\cap \bfN K^\times]},\end{equation}
		where the product runs over all places of $F$.
	\end{thm}
	
	\begin{proof}
		See \cite[Section 3]{Gra17} or \cite[Chapter IV]{Gra73}. Gras proved the theorem  for (narrow) ray class groups, but his proof works for class groups.
	\end{proof}
	
	\begin{rmk}
		$(1)$ The index $[\Lambda_D:\Lambda_D\cap \bfN K^\times]$ is independent of the choice of $D$.
		
		$(2)$ Take  $C=\{1\}$ and $D=\{1\}$, then  $\Lambda_D$ is the unit group $E_F$, and  Gras' formula is nothing but the ambiguous class number formula of Chevalley:
		\begin{equation}\label{chevalley}
		|\Cl_K^G|=|\Cl_F|\cdot\frac{\prod_{v}e_v}{[K:F]}\cdot\frac{1}{[E_F:E_F\cap \bfN K^\times]}.
		\end{equation}
		In fact the proof of Gras' formula is based on Chevalley's formula, whose proof can be found in \cite[Chapter 13, Lemma 4.1]{Lang90}.
	\end{rmk}

	One can use Hilbert symbols to compute  the index $[\Lambda_D:\Lambda_D\cap \bfN K^\times]$.

	\begin{lem}\label{lem: hasse_norm} Let $F$ be a number field and $\mu_d\subset F$. Assume $K=F(\sqrt[d]{a})$ is a Kummer extension of $F$ of degree $d$. Let $D$ be any subgroup of the group of fractional ideals of $K$ and $\Lambda_D=\{x\in F^\times\mid (x)\co_F \in \bfN D\}$.
		Define
		\[\rho=\rho_{D, K/F}:  \Lambda_D \longrightarrow \prod_{v}\zeta_{d},\quad x \mapsto \left(\hilbert{x}{a}{v}_d\right)_{v},\]
		where $v$ passes through all places of $F$ ramified in $K/F$. Then
		
		$(1)$ $\Ker(\rho) = \Lambda_D\cap \bfN K^\times$. In particular, $[\Lambda_D:\Lambda_D\cap \bfN K^\times]=|\rho(\Lambda_D)|$.
		
		$(2)$  Let $\Pi$ be the product map $\prod_{v}\mu_{d} \rightarrow \mu_d$, then $\Pi\circ \rho=1$ and hence
		$\rho(\Lambda_D)\subset \ker \Pi:= (\prod_{v }\zeta_{d})^{\Pi=1}$.
		
		$(3)$ $\Ker(\rho)$ and $|\rho(\Lambda_D)|$ are independent of the choice of $a$.
	\end{lem}

	\begin{proof}
		Let $I_K$ be the group of fraction ideals of $K$. Note that if $D\subset I_K$,   then $\Lambda_D\subset \Lambda:=\Lambda_{I_K}$.  Therefore it suffices to prove the results in the case $D=I_K$.
		
		(1) For $v$ a place of $F$, let $w$ be a place of $K$ above $v$. Recall that  $\hilbert{x}{a}{v}_d=1$ if and only if $x\in \bfN_{K_w/F_v}(K^\times_w)$.
		We claim that if $v$ is unramified, then $x\in \bfN_{K_w/F_v}(K^\times_w)$ for $x\in \Lambda$. Suppose $v$ is an infinite unramified place. Then $F_v=K_w$ and clearly $x\in \bfN_{K_w/F_v}(K^\times_w)$.  Suppose $v$ is a  finite  unramified place.
		Since  $x\in \Lambda$,  we have $(x)\co_F=\bfN(I)$. Then  locally $(x)\co_{F_v}=\bfN_{K_w/F_v}(J)$ for some  fractional ideal $J$ of $\co_{K_w}$. Since $\co_{K_w}$ is a principal ideal domain, $J=(\alpha)$ for some $\alpha \in K^\times_w$.  Hence $x=u\bfN_{K_w/F_v}(\alpha)$   with $u\in \co^\times_{F_v}$. Since $v$ is unramified, we have $u\in \bfN_{K_w/F_v}(K^\times_w)$ by local class field theory.  Therefore $x\in \bfN_{K_w/F_v}(K^\times_w)$.
		
		Now for $x\in \Ker (\rho)$, we have $x\in \bfN_{K_w/F_v}(K^\times_w)$ for every place $v$ of $F$.  Hasse's norm theorem \cite[Chapter VI, Corollary 4.5]{Neu13} gives $x\in \bfN K^\times$. So $\Ker (\rho) \subset \Lambda \cap \bfN K^\times$.   The other direction is clear.  This proves (1).

		(2) We have  proved that if $v$ is unramified, then $\hilbert{x}{a}{v}_d=1$ for $x\in \Lambda$. Therefore (2) follows from the product formula of Hilbert symbols.
		
		(3) is a consequence of (1).\end{proof}
	

	\section{Stability of $\ell$-class groups}
	
	We now give a stationary result about $\ell$-class groups in a finite cyclic $\ell$-extension. We first introduce the  ramification hypothesis $\RamHyp$. Let $F$ be a number field and $K$  an algebraic extension  (possibly infinite) of $F$. Then $K/F$ satisfies the  ramification hypothesis $\RamHyp$ if
	\begin{quote}
		Every place of $K$ ramified in $K/F$ is totally ramified in $K/F$ and there is at least one prime ramified in $K/F$. \end{quote}

	\begin{lem}
		Let $G$ be a  finite $\ell$-cyclic group with generator $\sigma$. Then $\Z_{\ell}[G]$ is a local ring with maximal ideal $(\ell,\sigma-1)$. 
	\end{lem}
	
	\begin{proof}
		Note that $\Z_{\ell}[G]\cong \Z_\ell[T]/(T^{\ell^n}-1)$ by sending $\sigma$ to $T$, where $\ell^n$ is the order of $G$. Let $\fm$ be a maximal ideal of $\Z_\ell[T]/(T^{\ell^n}-1)$. Then $\fm\cap \Z_{\ell}$ is a prime ideal of $\Z_\ell$. We claim that $\fm\cap \Z_\ell=\ell\Z_{\ell}$. 
		
		Otherwise $\fm\cap \Z_\ell=0$,  namely $\fm$ is disjoint with the multiplicative subset $\Z_{\ell}\setminus \{0\}$. Then $\fm$ corresponds to a prime ideal of the the ring $\Q_\ell[T]/{(T^{\ell^n}-1)}$. Each prime ideal of $\Q_\ell[T]/{(T^{\ell^n}-1)}$ is generated by a monic   irreducible polynomial $f(T)$ with $f(T)\mid T^{\ell^n}-1$.  By Gauss's lemma, $f(T)$ has $\Z_{\ell}$-coefficients.  Then $\fm=(f(T))$. But $\Z_{\ell}[T]/(f(T))$ is not a field since $\Z_{\ell}[T]/(f(T))$ is integral over $\Z_{\ell}$ and $\Z_\ell$ is not a field. So $\fm\cap \Z_{\ell}=\ell \Z_{\ell}$.
		
		Then $\fm$ corresponds to a maximal ideal of $\F_{\ell}[T]/(T^{\ell^n}-1)=\F_{\ell}[T]/(T-1)^{\ell^n}$. The latter is obviously a local ring with maximal ideal $(T-1)$. Hence $\fm= (\ell,T-1)$.  Therefore the maximal ideal of $\Z_{\ell}[G]$ is $(\ell,\sigma-1)$. \end{proof}
	
	\begin{prop}\label{prop:stable_theorem}
		Let $K_2/K_0$ be a cyclic extension of number fields of degree $\ell^2$ satisfying  $\RamHyp$. Let $K_1$ be the unique nontrivial intermediate field of  $K_2/K_0$.   Then for any $n\ge 1$,
		\[|\Cl_{K_0}/\ell^n\Cl_{K_0}|=|\Cl_{K_1}/\ell^n\Cl_{K_1}|\]
		implies that
		\[ \Cl_{K_2}/{\ell^n \Cl_{K_2}}\cong \Cl_{K_1}/{\ell^n \Cl_{K_1}}\cong \Cl_{K_0}/{\ell^n \Cl_{K_0}}.\]
		In particular, $|\Cl_{K_0}(\ell)|=|\Cl_{K_1}(\ell)|$ implies that $\Cl_{K_0}(\ell)\cong \Cl_{K_1}(\ell)\cong \Cl_{K_2}(\ell)$.
	\end{prop}
	\begin{proof}
		Denote by $G=\Gal(K_2/K_0)=\langle\sigma \rangle.$
		Let $L_i$ be the maximal unramified abelian $\ell$-extension of $K_i$ and $X_i=\Gal(L_i/K_i)$. By class field theory $X_i\cong \Cl_{K_i}(\ell)$.  By the maximal property, $L_2/K_0$ is a Galois extension. Let $\widetilde{G}:=\Gal(L_2/K_0)$. 
		The Galois group $G$ acts on $X:=X_2$ via $x^\sigma=\wt\sigma x\wt\sigma^{-1}$ where $\wt\sigma\in \widetilde{G}$ is any lifting of $\sigma$.  By this action $X$ becomes a module over the local ring $\Z_\ell[G]$.   Since $K_0\subset K_1\subset K_2$ satisfies  $\RamHyp$, we have $L_0\cap K_2=K_0$.   Then  $X/M= \Gal(K_2L_0/K_2) \cong X_0$ where $M=\Gal(L_2/K_2L_0)$.  Note that $K_2L_0/K_0$ is Galois,   so  $M$ and $X/M$ are  also  $\Z_\ell[G]$-modules. 
		We have the following claim:
		
		\vskip 0.3cm
		\noindent{\textbf{Claim}}:  $X/{\omega M}\cong X_1$,  where $\omega=1+\sigma+\cdots+\sigma^{\ell-1}\in \Z_\ell[G]$.
		\vskip 0.3cm
		
		Now for any $n\geq 1$, by the claim,
		\[ X_0/{\ell^n X_0}\cong \frac{X}{ M+\ell^nX} \text{ and } X_1/{\ell^n X_1}\cong \frac{X}{ \omega M+\ell^nX}.\]
		By the assumptions,  $  M+\ell^nX=\omega M+\ell^nX$.
		Since $\omega$ lies in the maximal ideal of $\Z_\ell[G]$,  we have $M\subset \ell^n X$ by Nakayama's Lemma. Hence  we have isomorphisms  which are induced by the restrictions
		\[ X/{\ell^n X} \cong  X_1/{\ell^nX_1}\cong X_0/{\ell^nX_0}.\]
		By class field theory we have isomorphisms which are induced by the norm maps
		\[ \Cl_{K_2}/{\ell^n \Cl_{K_2}}\cong \Cl_{K_1}/{\ell^n \Cl_{K_1}}\cong \Cl_{K_0}/{\ell^n \Cl_{K_0}}.\]
		Let $n\rightarrow +\infty$, we get $\Cl_{K_2}(\ell)\cong\Cl_{K_1}(\ell) \cong \Cl_{K_0}(\ell)$.

		Let us prove the claim.   Note that  $G=\widetilde{G}/X$. Let $\{\fp_1,\cdots,\fp_s\}$ be the set of places of $K_0$ ramified in $K_2/K_0$.  Note that   $\fp_i$ is not an infinite place by $\RamHyp$.  For each $\fp_i$, choose a prime ideal $\widetilde{\fp}_i$ of $L_2$ above  $\fp_i$. Let $I_i\subset \widetilde{G}$ be the inertia subgroup of $\widetilde{\fp}_i$. The map $I_i \hookrightarrow \widetilde{G} \twoheadrightarrow G$ induces an isomorphism $I_i\cong G$, since $L_2/K_2$ is unramified and $K_2/K_0$ is totally ramified. Let $\sigma_i\in I_i$ such that $\sigma_i \equiv \tilde{\sigma} \bmod X$. Then $I_i=\langle \sigma_i \rangle$.   Let $a_i=\sigma_i \sigma^{-1}_1 \in X$. Then $\langle I_1,\cdots,I_t\rangle = \langle \sigma_1,a_2,\cdots, a_t\rangle $. Since $L_0$ is the maximal unramified abelian $\ell$-extension of $K_0$, we have	
		\[\Gal(L_2/L_0)=\langle \widetilde{G}', I_1,\cdots,I_t\rangle =
		\langle \widetilde{G}', \sigma_1, a_2,\cdots, a_t\rangle\]
		where $\widetilde{G}'$ is the commutator subgroup of $\widetilde{G}$. In fact $\widetilde{G}'={(\sigma-1)}X$.  The inclusion  ${(\sigma-1)}X\subset \widetilde{G}'$ is clear. On the other hand, it is easy to check that  ${(\sigma-1)}X$ is normal in $\widetilde{G}$ and $X/{{(\sigma-1)}X}$ is in the center of $\widetilde{G}/{(\sigma-1)}X$. Since $\widetilde{G}/X\cong G$ is cyclic, from the exact sequence	
		\[ 1 \rightarrow X/{(\sigma-1)}X \rightarrow \widetilde{G}/{{(\sigma-1)}X} \rightarrow  G \rightarrow 1,\]
		we obtain $\widetilde{G}/{{(\sigma-1)}X}$ is abelian.  Thus we have  
		\[\Gal(L_2/L_0)= \langle  {(\sigma-1)}X,\sigma_1,a_2,\cdots, a_t\rangle.\]
		Since $a_i\in X$ and $X\cap I_1=\{1\}$,  we have $X\cap \Gal(L_2/L_0)= \langle {(\sigma-1)}X,a_2,\cdots, a_t\rangle$. Thus the map $X\hookrightarrow \widetilde{G}$ induces the following isomorphism
		\[ X/\langle {(\sigma-1)}X,a_2,\cdots, a_t\rangle \cong \widetilde{G}/\Gal(L_2/L_0) = X_0. \]
		Therefore $\langle {(\sigma-1)}X,a_2,\cdots, a_t\rangle=M$.  Repeat the above argument to $L_2/K_1$, we  obtain
		\[X/\langle {(\sigma^\ell-1)}X,b_2,\cdots, b_t\rangle  \cong X_1,\]
		where $b_i=\sigma^{\ell}_i\sigma^{-\ell}_1$ for each $i$.  Obviously, ${(\sigma^\ell-1)}X= {\omega (\sigma-1)}X$.  Recall that  $\sigma_i$ is a lifting of $\sigma$ so  by definition $x^{\sigma}=\sigma_i x \sigma^{-1}_i$ for $x\in X$. We have
		\[ \begin{split}
		b_i = \sigma^{\ell}_i \sigma^{-\ell}_1= \sigma^{\ell-1}_i a_i \sigma^{-(\ell-1)}_1=\sigma^{\ell-2}_i a_i \sigma_1 a_i \sigma^{-1}_1 \sigma^{-(\ell-2)}_1 &\\= \sigma^{\ell-2}_i {a_i}^{1+\sigma}\sigma^{-(\ell-2)}_1  = \cdots = a^{1+\sigma+\cdots +\sigma^{\ell-1}}_i=\omega a_i.
		\end{split}\]
		So  $\langle ({\sigma^\ell-1})X, b_2,\cdots, b_t\rangle=\omega M$ and then $X_1=X/{\omega M}$. This finishes the proof of the claim. \end{proof}
	
	\begin{rmk}\label{rmk:stable}			
		$(1)$ Let $K_\infty/K$ be a $\Z_\ell$-extension and $K_n$  its $n$-th layer. It is well known there exists $n_0 $  such that  $K_\infty/K_{n_0}$  satisfies  $\RamHyp$.  Then Proposition~\ref{prop:stable_theorem} recovers  Fukuda's result \cite{Fuk94}    that if  $|\Cl_{K_m}(\ell)|=|\Cl_{K_{m+1}}(\ell)|$ (resp.$ |\Cl_{K_m}/{\ell\Cl_{K_m}}|=|\Cl_{K_{m+1}}/{\ell\Cl_{K_{m+1}}}|$) for some $m\geq n_0$,  then  $|\Cl_{K_m}|=|\Cl_{K_{m+i}}|$ (resp. $ |\Cl_{K_m}/{\ell\Cl_{K_m}}|=|\Cl_{K_{m+1}}/{\ell\Cl_{K_{m+1}}}|$) for any $i\geq 1$.    In fact, our  proof is essentially the same  as  the proof of the corresponding results for $\Z_\ell$-extensions,  see \cite[Lemma 13.14, 13.15]{Was97} and \cite{Fuk94}.

		$(2)$ Let $K$ be a number field containing $\zeta_{\ell^2}$.  Let $a\in K^\times\backslash K^{\times \ell}$ and $K_{n}=K(\sqrt[\ell^n]{a})$. Then $\Gal(K_{m+2}/K_m)\cong \Z/{\ell^2\Z}$ for any $m$.  One can show that  there exists some $n_0$ such that  $K_\infty/K_{n_0}$ satisfies  $\RamHyp$.   If  $|\Cl_{K_m}(\ell)|=|\Cl_{K_{m+1}}(\ell)|$ for some $m\geq n_0$,  then by repeatedly applying Proposition~\ref{prop:stable_theorem}, one can get $|\Cl_{K_{m+i}}(\ell)|=|\Cl_{K_m}(\ell)|$ for any $i\geq 0$.
	\end{rmk}

	Now let $\ell$ and $p$ be   prime numbers and  $K_{n,m}=\Q(p^{\frac{1}{\ell^n}},\zeta_{2\ell^{m}})$. The following result is a consequence of Proposition \ref{prop:stable_theorem}.

	\begin{prop}\label{prop: square_clgp} Assume that all the primes above $\ell$  in $K_{n_0,m_0}$ are  totally ramified in $K_{n_0+1,m_0+1}$  for some integers $n_0\geq 0$ and $m_0\geq 1$ if $\ell\neq 2$ or  $n_0\geq  v_p(2)$ and $m_0\geq 1+v_p(2)$ if $\ell=2$. Then
		\begin{enumerate}
			\item All primes above $\ell$ in $K_{n_0,m_0}$ are totally ramified in $K_{n,m}/K_{n_0,m_0}$ for  all $(n,m)\geq (n_0, m_0)$;
			\item If $|A_{n_0,m_0}|=|A_{n_0+1,m_0+1}|$, then $A_{n,m}\cong A_{n_0,m_0 }$ for all $(n,m)\geq (n_0,m_0)$.
			
			\item  If $\ell \nmid h_{n_0+1,m_0+1}$, then $\ell \nmid h_{n,m}$ for all $(n,m)\geq (n_0, m_0)$.
		\end{enumerate}
	\end{prop}

	\begin{proof} By the assumption for $n_0$ and $m_0$, one has $[K_{n_0+1,m_0+1}:K_{n_0,m_0}]=\ell^2$ and  
		\[ \begin{split}
		\Gal(K_{n_0,m_0+2}/K_{n_0,m_0}) &\cong  \Gal(K_{n_0+1,m_0+2}/K_{n_0+1,m_0}) \\ 
		&\cong  \Gal(K_{n_0+2,m_0+2}/K_{n_0,m_0+2}) \cong \Z/{\ell^2\Z}.\end{split}\]
		Consider the diagram. 
		\[ \begin{xymatrix}{
			K_{n_0,m_0+2} \ar@{-}[d] \ar@{-}[r] & K_{n_0+1,m_0+2} \ar@{-}[d] \ar@{-}[r] & K_{n_0+2,m_0+2}\\
			K_{n_0,m_0+1}\ar@{-}[r] & K_{n_0+1,m_0+1}  &   \\
			K_{n_0,m_0}\ar@{-}[u]\ar@{-}[r] & K_{n_0+1,m_0}\ar@{-}[u]    & \\
		}\end{xymatrix} \]
		
		For (1),
		let  $\fq$ be a prime of $K_{n_0,m_0}$ above $\ell$.  Apply   Lemma \ref{lem: ram} to the two vertical lines in the diagram,  we obtain $\fq$ is totally ramified in $K_{n_0+1,m_0+2}/K_{n_0,m_0}$. Apply  Lemma \ref{lem: ram} to the  top horizontal  line in the diagram, we get  $\fq$ is totally ramified in $K_{n_0+2,m_0+2}/K_{n_0+2,m_0}$. Hence $\fq$ is totally ramified in $K_{n_0+2,m_0+2}/K_{n_0,m_0}$.  Repeatedly using the above argument, we obtain $\fq$ is totally ramified in $K_{n,m}/K_{n_0,m_0}$ for  all $n\geq n_0$ and $m\geq m_0$.

		For (2), by  Lemma~\ref{prop: norm_surj}, $|A_{n_0,m_0}|=|A_{n_0+1,m_0+1}|$ implies that
		\[A_{n_0+1,m_0+1}\cong A_{n_0+1,m_0}\cong A_{n_0,m_0+1}\cong A_{n_0,m_0}.\]
		If $p=\ell$,   the two vertical  lines and the  top horizontal  line  in the diagram   satisfy $\RamHyp$ by (1).  If $p\neq \ell$,
		let $\fp$ be a prime of $K_{0,m}$ above $p$. For any $n\geq 1$,  note that  $x^{\ell^n}-p$ is a  $\fp$-Eisenstein polynomial in $K_{0,m}[x]$. Therefore $K_{n,m}/K_{0,m}$ is totally ramified at $\fp$ for each $n,m$. In particular the horizontal line is totally ramified at $\fp$.  Since $K_{\infty,\infty}/K_{n_0,m_0}$ is unramified outside $\ell$ and $p$,  the two horizontal lines and the right most vertical line  in the diagram all satisfy $\RamHyp$ by (1).  
		
		Since $K_{n_0,m_0+2}/K_{n_0,m_0}$  is a cyclic extension of degree $\ell^2$, applying Proposition \ref{prop:stable_theorem} to this extension, we get 
		\[ A_{n_0,m_0+2}\cong A_{n_0,m_0+1}\cong A_{n_0,m_0}.\]
		Similarly, applying Proposition \ref{prop:stable_theorem}   to  $K_{n_0+1,m_0+2}/K_{n_0+1,m_0}$, we obtain 
		\[ A_{n_0+1,m_0+2}\cong A_{n_0+1,m_0+1}\cong A_{n_0+1,m_0}.\]	
		Therefore $A_{n_0+2,m_0+1}\cong A_{n_0+2,m_0}$. Note that $K_{n_0+2,m_0+2}/K_{n_0,m_0+2}$ is also a cyclic extension of degree $\ell^2$. Applying Proposition \ref{prop:stable_theorem}  to this extension, we obtain
		\[A_{n_0+2,m_0+2}\cong A_{n_0+1,m_0+1} \cong A_{n_0,m_0+2}.\]
		Thus  $A_{n_0+2,m_0+2}\cong A_{n_0+1,m_0+1}$. 
		Using the above argument inductively, we  have  $A_{n_0+k,m_0+k}\cong A_{n_0,m_0 }$ for  any $k\geq 1$. Finally we have $A_{n,m}\cong A_{n_0,m_0}$ by Lemma~\ref{prop: norm_surj}.

		For (3), $\ell \nmid h_{n_0+1,m_0+1}$ implies that $\ell \nmid h_{n_0,m_0}$ by Lemma~\ref{prop: norm_surj}. Then the result follows from (2).  \end{proof}

	\section{The case that $\ell$ is odd}

	\begin{lem}\label{lem: ram_odd}
		Assume $p$ is either $\ell$ or a primitive element modulo  $\ell^2$. Then $\ell$ is totally ramified in $K_{n,m}$ for any $(n, m)>(0,0)$. 
	\end{lem}
	
	\begin{proof} For $n \geq 1$,  $(x+p)^{\ell^n}-p$ is  an Eisenstein polynomial in $\Q_\ell[x]$ by the assumptions on $p$ and $\ell$, hence is irreducible in $\Q_\ell[x]$. This means that the extension $\Q_\ell(p^{\frac{1}{\ell^n}})/\Q_\ell$ is totally ramifield of degree $\ell^n$ and $\mu_\ell \not\subset\Q_\ell(p^{\frac{1}{\ell^n}})$. As a result $\Q_\ell(p^{\frac{1}{\ell^n}})/\Q_\ell(p^{\frac{1}{\ell^{n-1}}})$ is non-Galois of degree $\ell$,  one has $\Q_\ell(p^{\frac{1}{\ell^n}}, \zeta_{\ell^m})/\Q_\ell(p^{\frac{1}{\ell^{n-1}}}, \zeta_{\ell^m})$ is also of degree $\ell$. By induction,  
		\[ [\Q_\ell(p^{\frac{1}{\ell^n}}, \zeta_{\ell^m}): \Q_\ell]= \ell\cdot [\Q_\ell(p^{\frac{1}{\ell^{n-1}}}, \zeta_{\ell^m}): \Q_\ell]=\ell^n (\ell^m-\ell^{m-1}).   \]
		Then the extension  $\Q_\ell(p^{\frac{1}{\ell^n}},\zeta_{\ell^{n}})/\Q_\ell(\zeta_{\ell^{n}})$ is cyclic of degree $\ell^n$, with the only subextensions of the form  $\Q_\ell(p^{\frac{1}{\ell^k}},\zeta_{\ell^{n}})$ for $0\leq k\leq n$. 
		If $\Q^{\ab}_\ell \cap \Q_\ell(p^{\frac{1}{\ell^n}},\zeta_{\ell^{n}})\supsetneq \Q_\ell(\zeta_{\ell^n})$, then  there exists $k>0$ such that $p^{\frac{1}{\ell^k}}\in \Q_\ell^{\ab}$ and hence   $p^{\frac{1}{\ell}}\in \Q_\ell^{\ab}$, impossible. Hence $\Q^{\ab}_\ell \cap \Q_\ell(p^{\frac{1}{\ell^n}},\zeta_{\ell^{n}})=\Q_\ell(\zeta_{\ell^n})$. Thus $\ell$ is totally ramified in $K_{n,n}$ for any $n\geq 1$, and therefore totally ramified in $K_{n,m}$ for all $(n,m)>(0,0)$.
	\end{proof}

	\begin{proof}[Proof of Theorem~\ref{thm: odd}]
		By Proposition~\ref{prop: square_clgp} and Lemma~\ref{lem: ram_odd}, if $\ell\nmid h_{1,2}$, then $\ell \nmid h_{n,m}$  for any $(n,m)\geq (1,2)$ and then $\ell\nmid h_{n,m}$ for any $(n,m)\geq (0,0)$  by Lemma~\ref{prop: norm_surj}. We prove $\ell\nmid h_{1,2}$  by  applying Chevalley's formula \eqref{chevalley} to $K_{1,2}/K_{0,2}$.  We treat the case $p\neq \ell$ and leave the case $p=\ell$ to the readers.
		
		Since $p$ is inert in $K_{0,2}$,  the ramified primes in $K_{1,2}/K_{0,2}$ are $p\co_{0,2}$ and $(1-\zeta_{\ell^2})\co_{0,2}$. As $\ell$ is regular, one has   $\ell$ does not divides the class number $K_{0,m}$ for any $m\geq 1$, see \cite[Corollary 10.5]{Was97}.  We now calculate the unit index in Chevalley's formula.  Recall  the following map as in Lemma~\ref{lem: hasse_norm}:
		\[\fct{\rho: E_{0,2}}{\mu_\ell \times \mu_\ell}{x}{\left(\hilbert{x}{p}{p\co_{0,2}}_\ell, \hilbert{x}{p}{(1-\zeta_{\ell^2})}_\ell \right).}\]
		We have the index $[E_{0,2}:E_{0,2}\cap \bfN K^\times_{0,2}]=|\rho(E_{0,2})| \leq \ell$ by product formula.  Since $p$ is a primitive root modulo $\ell^2$, we have $\ell^2\nmid p^{\ell-1}-1$. Then by the norm-compatibility of the Hilbert symbols, 
		\[ \hilbert{\zeta_{\ell^2}}{p}{p\co_{0,2}}_\ell = \hilbert{\zeta_{\ell}}{p}{p\co_{0,1}}_\ell = \zeta^{\frac{p^{\ell-1}-1}{\ell}}_{\ell}\neq 1.\] 
		Thus $|\rho(E_{0,2})| = \ell$ and Chevalley's formula gives  $\ell \nmid |\Cl^{G}_{1,2}|$ where $G=\Gal(K_{1,2}/K_{0,2})$. Therefore $\ell \nmid h_{1,2}$ by Lemma~\ref{nakayama}.  \end{proof}

	\begin{proof}[Proof of Theorem~\ref{thm:main2}]
		(1) is a  special case of Theorem~\ref{thm: odd}. 
		
		For (2), by tracing the proof of Lemma~\ref{lem: ram_odd}, we obtain that $3$ is totally ramified  in $K_{n,n}/\Q$ for any $n\geq 1$.  To prove (2), we first show that $A_{2,2}\cong A_{1,1}\cong \Z/{3\Z}$.  We   apply Gras' formula \eqref{eq:Gras} in the case
		\[K_{2,2}/K_{0,2},\ C=\langle \cl(\fq_{2,2}) \rangle,\ D=\langle  \fq_{2,2} \rangle \]
		where $\fq_{2,2}$ is the unique prime ideal  of $K_{2,2}$ above $3$.  In this case
		\[\Lambda_D=\langle \pm \zeta_9, 1-\zeta_9, 1-\zeta^2_9, 1-\zeta^4_9\rangle.\] 
		Since $p\equiv 4,7\bmod 9$, we have $p\co_{0,2}=\fp_1\fp_2$.  The ramified primes of $K_{0,2}$ in $K_{2,2}$ are $\fq_{0,1},\fp_1,\fp_2$.  For the map 
		\[\fct{\rho: \Lambda_D}{\mu_9 \times \mu_9 \times \mu_9}{x}{\left(\hilbert{x}{p}{\fp_1}_9, \hilbert{x}{p}{\fp_2}_9,\hilbert{x}{p}{\fq_{0,2}}_9\right)}\]
		defined in Lemma \ref{lem: hasse_norm}, we know $ \rho(\Lambda_D) \subset (\mu_9 \times \mu_9 \times \mu_9)^{\prod=1}$,
		$[\Lambda_D : \Lambda_D\cap \bfN(K^\times_{2,2})]=|\rho(\Lambda_D)|$ and $[E_{0,2} : E_{0,2}\cap \bfN(K^\times_{2,2})]=|\rho(E_{0,2})|$. 
		
		Now Lemma~\ref{lem:3part} tells us that  $|\rho(\Lambda_D)|=81$ and $|\rho(E_{0,2})|=27$.
		Hence Gras' formula implies that $3\nmid (\Cl_{2,2}/C)^G$ where $G=\Gal(K_{2,2}/K_{0,2})$. This means $A_{2,2}=C$ by Lemma~\ref{nakayama}. In particular, $A_{2,2}=\Cl^G_{2,2}(3)$. By Chevalley's formula \eqref{chevalley}, we have $|A_{2,2}|=|\Cl^G_{2,2}|=3$.  For $m\leq 2, n\leq 2$,  the norm map from $A_{2,2}$ to $A_{m,n}$ is surjective. It has been shown In \cite{Aouissi} that $A_{1,0}\cong \Z/{3\Z}$, the inequalities  $|A_{1,0}|\leq |A_{1,1}|\leq |A_{2,2}|$ then imply that $A_{2,2}\cong A_{1,1}\cong \Z/{3\Z}$. 
		
		By Proposition~\ref{prop: square_clgp},  we have $A_{n,m}\cong \Z/{3\Z}$ for any $n\geq 1, m\geq 1$. For  $n\geq 1$,  note that   $ 3=|A_{1,0}|\leq |A_{n,0}|\leq |A_{n,1}|=3$, then $A_{n,0}\cong \Z/{3\Z}$. This completes the proof of (2). \end{proof}

	\begin{lem} \label{lem:3part}
		We have $|\rho(\Lambda_D)|=81$ and $|\rho(E_{0,2})|=27$.
	\end{lem}

	\begin{proof}
		By the product formula, $|\rho(\Lambda_D)|\leq 81$. To get equality, it suffices to show  $|\rho(\Lambda_D)|\geq 81$.

		We first compute $\rho(\zeta_9)$. In the local field $\Q_p(\zeta_9)$, one has
		\[\hilbert{\zeta_9}{p}{\Q_p(\zeta_9)}_9=\zeta^{\frac{p^3-1}{9}}_9  \]
		which is a primitive $9$-th root of unity since $p\equiv 4,7\bmod 9$. 
		The prime ideals $\fp_1$ and $\fp_2$ above $p$ induce two  embeddings from $K_{0,2}$ to $\Q_p(\zeta_9)$ which are not $\Gal(\overline{\Q}_p/\Q_p)$-conjugate. We choose the corresponding embeddings  by setting  $\fp_1(\zeta_9)=\zeta_9$ and $\fp_2(\zeta_9)=\zeta^{-1}_9$.  Then
		\[ \hilbert{\zeta_9}{p}{\fp_1}_9=\hilbert{\zeta_9}{p}{\fp_2}_9 = \zeta^{\frac{p^3-1}{9}}_9.\]
		By the product formula, one has \[\rho(\zeta_9)=(\zeta^{\frac{p^3-1}{9}}_9,\zeta^{\frac{p^3-1}{9}}_9,\zeta^{-\frac{2(p^3-1)}{9}}_9) \text{ and } |\langle \rho(\zeta_9) \rangle| =9.\]
		
		To prove $|\rho(\Lambda_D)|\geq 81$, it suffices to show that $\rho(1-\zeta_9)^3 \not\in \langle \rho(\zeta_9) \rangle$. 
		We have
		\[ \hilbert{1-\zeta_9}{p}{\Q_p(\zeta_9)}_9^3=\hilbert{1-\zeta_9}{p}{\Q_p(\zeta_9)}_3=\hilbert{1-\zeta_3}{p}{\Q_p}_3, \]
		and hence 
		\[ \begin{split}
		\hilbert{1-\zeta_9}{p}{\fp_1}_9^3   \hilbert{1-\zeta_9}{p}{\fp_2}_9^3   =&\hilbert{1-\zeta_9}{p}{\Q_p(\zeta_9)}_3\hilbert{1-\zeta^{-1}_9}{p}{\Q_p(\zeta_9)}_3\\ =&\hilbert{1-\zeta_3}{p}{\Q_p}_3\hilbert{1-\zeta^{-1}_3}{p}{\Q_p}_3=\hilbert{3}{p}{\Q_p}_3\neq 1, \end{split}\]
		where the first equality is by definition, the second equality is by the norm-compatibility  of Hilbert symbols, and the last equality is by assumptions on $p$.  This implies  $\hilbert{1-\zeta_9}{p}{\fp_1}^3_9\neq \hilbert{1-\zeta_9}{p}{\fp_2}^3_9$ and $\rho(1-\zeta_9)^3 \not\in \langle \rho(\zeta_9)\rangle$.
		
		Now we compute $|\rho(E_{0,2})|$.  Since $3\mid h_{1,0}$, one has $3\mid h_{2,2}$ by Lemma~\ref{prop: norm_surj}. By Chevalley's formula and Lemma~\ref{nakayama}, we must have 
		\[
		|\rho(E_{0,2})|\leq 27.
		\]
		Let $\sigma_4\in \Gal(\Q(\zeta_8)/\Q)$ be given by $\sigma_4(\zeta_9)=\zeta^4_9$. Since $p\equiv 4,7\bmod 9$, we have  $\sigma_4(\fp_i)=\fp_i$ ($i=1,2$). It follows then
		\[ \hilbert{1-\zeta^4_9}{p}{\fp_i}_9\equiv (1-\zeta^4_9)^{\frac{p^3-1}{9}} \bmod \fp_i =\sigma_4\left(\hilbert{1-\zeta_9}{p}{\fp_i}_9\right) =\hilbert{1-\zeta_9}{p}{\fp_i}^4_9.\] 
		Therefore $\rho(\frac{1-\zeta^4_9}{1-\zeta_9})=\rho(1-\zeta_9)^3$. As we have proved, $|\rho(E_{0,2})|\geq |\langle \rho(\zeta_9),\rho(\frac{1-\zeta^4_9}{1-\zeta_9}) \rangle|=27$. Hence $|\rho(E_{0,2})|=27$.
	\end{proof}

	\section{The case $\ell=2$}
	In this section,   $K_{n,m}=\Q(p^{\frac{1}{2^n}},\zeta_{2^{m+1}})$,  $A_{n,m}$ and $h_{n,m}$ are the $2$-part of the class group and the class number of $K_{n,m}$ respectively.

	\subsection{The cases $p=2$ and $p\equiv 3,5\bmod 8$ }

	\begin{proof}[Proof of Theorem~\ref{thm:main1} for $p=2$]	
		The prime $2$ is totally ramified in $K_{2,3}=\Q(\sqrt[4]{2},\zeta_{16})$ and $h_{2,3}=1$.  Therefore $2$ is totally ramified in $K_{\infty,\infty}$  and $2\nmid h_{n,m}$ for $n\geq 1, m\geq 2$ by Proposition~\ref{prop: square_clgp}. The remaining $(n,m)$ follows from Lemma~\ref{prop: norm_surj}.
	\end{proof}

	\begin{lem}\label{hil_sym_3mod8}
		Suppose $p\equiv 3 \bmod 8$.
		
		$(1)$ The unique prime above $2$ in $K_{1,1}$ is totally ramified in $K_{\infty,\infty}/K_{1,1}$.
		
		$(2)$ $\prod_{v}e_v=32$ where $v$ runs over the places of $K_{0,2}$ and $e_v$ is the ramification index of $v$ in $K_{2,2}/K_{0,2}$.
		
		$(3)$ $[E_{0,2}:E_{0,2}\cap \bfN K^\times_{2,2}]=8$.
	\end{lem}
	\begin{proof}
		(1)	We only need to show  that the unique prime above $2$ in $K_{1,1}$ is totally ramified in $K_{2,2}/K_{1,1}$  by Proposition \ref{prop: square_clgp}.

		It is easy to see that $K_{1,2}/K_{1,1}$ is  ramified at the prime above $2$. To see the prime above $2$ is also ramified in  $K_{2,2}/K_{1,2}$,  we consider the local fields extension $\Q_2(\zeta_8,\sqrt[4]{p})/\Q_2(\zeta_8,\sqrt{p})$.   Note that
		\[ \Q_2(\sqrt[4]{p}) =
		\begin{cases}
		\Q_2(\sqrt[4]{3})  & \text{ if } p\equiv 3 \bmod 16, \\
		\Q_2(\sqrt[4]{11})  & \text{ if } p\equiv 11 \bmod 16. \\
		\end{cases} \]
		Since the fields $\Q_2(\sqrt[4]{3})$ and $\Q_2(\sqrt[4]{11})$ are not Galois over $\Q_2$,
		\[ \Q^{\ur}_2 \cap \Q_2(\zeta_8,\sqrt[4]{p}) \subset \Q^{\ab}_2\cap \Q_2(\zeta_8,\sqrt[4]{p})=\Q_2(\zeta_8,\sqrt{p}), \]
		where $\Q^{\ur}_2$ (resp. $\Q^{\ab}_2$) is the maximal unramified  (resp. abelian) extension of $\Q_2$.
		Thus $\Q_2(\zeta_{8},\sqrt[4]{p})/\Q_2(\zeta_{8},\sqrt{p})$ is totally ramified.  So $K_{2,2}/K_{1,1}$ is totally ramified at $2$.
		
		(2) Since $p\equiv 3 \bmod 8$, we have  $p\co_{0,2}=\fp_1\fp_2$, with   $\fp_1,\fp_2$ totally ramified in $K_{\infty,2}$. Then $e_{\fp_i}=[\Q_p(\sqrt[4]{p},\zeta_8): \Q_p(\zeta_8)]=4$.
		Let $\fq$ be the unique prime ideal above $2$ in $K_{0,2}$. Then $e_{\fq}=2$ as $\Q_2(\sqrt{p},\zeta_{8})/\Q_2(\zeta_{8})$ is unramified. Since $K_{2,2}/K_{0,2}$ is unramified outside $2$ and $p$, we have $\prod_{v}e_v =32$.
		
		(3) Note that  $E_{0,2}=\pair{ \zeta_8,1+\sqrt{2} }$. Recall  the following map as in Lemma~\ref{lem: hasse_norm}:
		\[\fct{\rho: E_{0,2}}{\mu_4 \times \mu_4 \times \mu_4}{x}{\left(\hilbert{x}{p}{\fp_1}_4, \hilbert{x}{p}{\fp_2}_4,\hilbert{x}{p}{\fq}_4\right).}\]
		
		We have  $|\rho(E_{0,2})|=[E_{0,2}:E_{0,2}\cap \bfN K^\times_{2,2}]$ and  $\rho(E_{0,2}) \subset (\mu_4\times \mu_4\times\mu_4) ^{\prod=1}$.
		
		Let $\iota_1,\iota_2: \Q(\zeta_8)\rightarrow \Q_p(\zeta_8)$  be the  corresponding embeddings  of $\fp_1,\fp_2$ such that  $\iota_1(\zeta_8)=\zeta_8$ and $\iota_2(\zeta_8)=\zeta^{-1}_8$.
		By definition  $\hilbert{x}{p}{\fp_j}_4=\iota^{-1}_j \hilbert{\iota_j(x)}{p}{\Q_p(\zeta_8)}_4$ for $j=1,2$.
		
		We first compute $\rho(\zeta_8)$.
		Since the residue field of $\Q_p(\zeta_8)$ is $\F_{p^2}$, we have
		\[\hilbert{\zeta_8^{\pm1}}{p}{\Q_p(\zeta_8)}_4=\hilbert{p}{\zeta_8^{\pm1}}{\Q_p(\zeta_8)}^{-1}_4=
		\zeta^{\mp\frac{p^2-1}{4}}_8.\]
		Thus
		\[ \hilbert{\zeta_8}{p}{\fp_1}_4=\hilbert{\zeta_8}{p}{\fp_2}_4=\zeta^{-\frac{p^2-1}{4}}_8= \pm i. \]
		By the product formula $\hilbert{\zeta_8}{p}{\fq}_4=-1$. Therefore $\rho(\zeta_8)=(\pm i,\pm i, -1)$.
		
		Now we compute $\rho(1+\sqrt{2})$. In the local field $\Q_p(\zeta_8)$,
		\[ \hilbert{1+\sqrt{2}}{p}{\Q_p(\sqrt{2})}^2_4=\hilbert{1+\sqrt{2}}{p}{\Q_p
			(\sqrt{2})}_2=\hilbert{-1}{p}{\Q_p}_2=-1,\]
		Hence
		\[ \hilbert{1+\sqrt{2}}{p}{\Q_p(\sqrt{2})}_4=\pm i.\]
		Since $\iota_1(1+\sqrt{2})=\iota_2(1+\sqrt{2})=1+\sqrt{2}$ and $\iota_1(i)=i, \iota_2(i)=-i$, we  have
		\[ \hilbert{1+\sqrt{2}}{p}{\fp_1}_4=\pm i, \quad \hilbert{1+\sqrt{2}}{p}{\fp_2}_4=\mp i.\]
		By the product formula, $\hilbert{1+\sqrt{2}}{p}{\fq}_4=1$.
		
		Therefore, $\rho(\zeta_8)=(\pm i,\pm i,-1)$ and $\rho(1+\sqrt{2})=(\pm i, \mp i,1)$. In each case, we  have    $|\rho(E_{0,2})|=8$.
	\end{proof}

	\begin{proof}[Proof of Theorem~\ref{thm:main1} for $p\equiv3 \bmod 8$]
		We know  that the class number of $K_{0,2}=\Q(\zeta_8)$ is $1$, the product of the ramification indices is $32$ and the index $[E_{0,2}:E_{0,2}\cap \bfN K^\times_{2,2}]=8$ by Lemma~\ref{hil_sym_3mod8}, then $|\Cl^{G}_{2,2}|=1$ by Chevalley's formula~\eqref{chevalley}. Thus $2\nmid h_{2,2}$  by Lemma~\ref{nakayama}. Now Proposition~\ref{prop: square_clgp} implies   $2\nmid h_{n,m}$ for  $n,m\ge 1$.
		Since $K_{n,1}/K_{n,0}$ is ramified at the real places,  we have $2\nmid h_{n,0}$  by Lemma \ref{prop: norm_surj}.
	\end{proof}

	\begin{lem} Suppose  $p\equiv 5\bmod 8$.
		
		$(1)$ The unique prime in $K_{1,0}$ above $2$ is totally ramified in $K_{\infty,\infty}/K_{1,0}$.
		
		$(2)$ $\prod_{v}e(v,K_{3,2}/K_{0,2})=2^8$ where $v$ runs over the places of $K_{0,2}$.
		
		$(3)$ $\prod_{v}e(v,K_{2,1}/K_{0,1})=2^5$ where $v$ runs over the places of $K_{0,1}$.
		
		$(4)$ $\prod_{v}e(v,K_{1,2}/K_{0,2})=4$ where $v$ runs over the places of $K_{0,2}$.
	\end{lem}
	\begin{proof}	
		(1) Note that $\Q_2(\sqrt[4]{p})/\Q_2$ is not Galois, so  $\sqrt[4]{p}\notin \Q^{\ab}_2$. Then the proof is the same as the case $p\equiv 3 \bmod 8$.
		
		(2) We only need to consider the primes above $2$ and $p$.
		Since $e(p,K_{3,0}/\Q)=8$ and  $p\co_{0,2}=\fp_1\fp_2$,  we have $e(\fp_1,K_{3,2}/K_{0,2})=e(\fp_2,K_{3,2}/K_{0,2})=8$. From (1), we can easily obtain that $e(\fq_{0,2},K_{3,2}/K_{0,2})=4$ for $\fq_{0,2}$ the only prime above $2$ in $K_{0,2}$. Hence the product of ramification indexes is $2^8$.
		
		The proofs of (3) and (4)  is easy, we leave it to the readers.
	\end{proof}

	\begin{lem}\label{lem:5mod8 index}
		
		Let $p\equiv 5\bmod 8$. Let  $\Lambda_{0,2}=\pair{ (1-\zeta_8)^2, \zeta_8,1+\sqrt{2} } \subset K^\times_{0,2}$ and $\Lambda_{0,1}= \pair{ (1-i)^2, i } \subset K^\times_{0,1}$. We have
		
		$(1)$ $[\Lambda_{0,2} : \Lambda_{0,2}\cap \bfN K^\times_{3,2}]=32$ and  $[E_{0,2}:E_{0,2}\cap \bfN K^\times_{3,2}]=16$;
		
		$(2)$ $[\Lambda_{0,1} : \Lambda_{0,1}\cap \bfN K^\times_{2,1}]=8$ and $[E_{0,1}: E_{0,1}\cap \bfN K^\times_{2,1}]=4$;
		
		$(3)$ $[E_{0,2}:E_{0,2}\cap \bfN K^\times_{1,2}]=2$.
	\end{lem}
	\begin{proof}
		Denote by $\fq_{n,m}$ the unique prime ideal of $K_{n,m}$ above $2$ for each $n,m \geq 0$. 	   Note that  $E_{0,2}=\pair{ \zeta_8,1+\sqrt{2} }$.   Then $\Lambda_{0,2}=\Lambda_{\langle \fq_{3,2} \rangle}$ corresponds to the extension $K_{3,2}/K_{0,2}$ and $\Lambda_{0,1}=\Lambda_{\langle \fq_{2,1} \rangle}$ corresponds to the extension $K_{2,1}/K_{0,1}$ as in Lemma \ref{lem: hasse_norm}.
		
		Since $p\equiv 5\bmod 8$, we have $p\co_{0,1}=\fp_1\fp_2$ and $p\co_{0,2}=\fP_1\fP_2$. Note that $\fP_1,\fP_2,\fq_{0,2}$ are exactly the ramified places in $K_{3,2}/K_{0,2}$.   For (1), we study the map as in Lemma \ref{lem: hasse_norm}:
		\[\fct{\rho:=\rho_{\langle \fq_{3,2}\rangle, K_{3,2}/K_{0,2}}: \Lambda_{0,2}}{\mu_8 \times \mu_8 \times \mu_8}{x}{\left(\hilbert{x}{p}{\fP_1}_8, \hilbert{x}{p}{\fP_2}_8,\hilbert{x}{p}{\fq_{0,2}}_8\right).}\]
		By Lemma \ref{lem: hasse_norm}, $ \rho(\Lambda_{0,2}) \subset (\mu_8 \times \mu_8 \times \mu_8)^{\prod=1}$,
		$[\Lambda_{0,2} : \Lambda_{0,2}\cap \bfN(K^\times_{3,2})]=|\rho(\Lambda_{0,2})|$ and $[E_{0,2} : E_{0,2}\cap \bfN(K^\times_{3,2})]=|\rho(E_{0,2})|$.
		
		Let  $ \iota_j:\Q(\zeta_8)\to\Q_p(\zeta_8)$  the corresponding embeddings for $\fP_j$  for $j=1,2$. We choose $ \iota_j$ so that  $\iota_1(\zeta_8)=\zeta_8$ (and hence $\iota(i)=i,\iota(\sqrt{2})=\sqrt{2}$) and $\iota_2(\zeta_8)=\zeta^{-1}_8$ (and hence $\iota_2(i)=-i,\iota_2(\sqrt{2})=\sqrt{2}$).
		The Hilbert symbol $\hilbert{x}{p}{\fP_i}_8$ by definition is  $\iota^{-1}_i  \hilbert{\iota_i(x)}{p}{\Q_p(\zeta_8)}_8$.

		We first compute $\rho(\zeta_8)$. In the local field $\Q_p(\zeta_8)$,
		\[\hilbert{\zeta^{\pm 1}_8}{p}{\Q_p(\zeta_8)}_8=\hilbert{p}{\zeta_8^{\pm}}{\Q_p(\zeta_8)}^{-1}_8=\zeta^{\mp \frac{p^2-1}{8}}_8, \]
		we have
		\[ \hilbert{\zeta_8}{p}{\fP_{1}}_8=\hilbert{\zeta_8}{p}{\fP_2}_8
		=\zeta^{-\frac{p^2-1}{8}}_8.\]
		Hence $\rho(\zeta_8)=(\zeta^{-\frac{p^2-1}{8}}_8,\zeta^{-\frac{p^2-1}{8}}_8,\pm i) $ by the product formula.

		Now we compute $\rho(1+\sqrt{2})$.  In $\Q_p(\zeta_8)$,
		\[\hilbert{1+\sqrt{2}}{p}{\Q_p(\zeta_8)}^2_8=
		\hilbert{1+\sqrt{2}}{p}{\Q_p(\zeta_8)}_4=\hilbert{-1}{p}{\Q_p}_4=-1,\]
		where the second equality is due to   the norm-compatible property of Hilbert symbols and the fact $i\in \Q_p$ for $p\equiv 5\bmod 8$,
		the last equality is due to the fact  $-1$ is a square  but not a fourth power  in $\Z/{p\Z}$ for  $p\equiv 5\bmod 8$.
		Therefore
		\[ \hilbert{1+\sqrt{2}}{p}{\Q_p(\zeta_8)}_8=\pm i. \]
		Since $\iota_1(\sqrt{2})=\iota_2(\sqrt{2})=\sqrt{2}$ and $\iota_1(i)=i, \iota_2(i)=-i$, we have
		\[\hilbert{1+\sqrt{2}}{p}{\fP_{1}}_8=\pm i, \quad \hilbert{1+\sqrt{2}}{p}{\fP_2}_8=\mp i.\]
		Hence $\rho(1+\sqrt{2})=(\pm i, \mp i,1)$ by the product formula.    In each case, we always have $|\rho(E_{0,2})|=|\langle \rho(\zeta_8), \rho(1+\sqrt{2}) \rangle|=16$.

		Finally we compute  $\rho((1-\zeta_8)^2)$. In $\Q_p(\zeta_8)$,
		\[a^\pm:=\hilbert{(1-\zeta_8^{\pm1})^2}{p}{\Q_p(\zeta_8)}_8=\hilbert{1-\zeta_8^{\pm1}}{p}{\Q_p(\zeta_8)}_4=\hilbert{(1-\zeta_8^{\pm1})(1+\zeta_8^{\pm1})}{p}{\Q_p}_4=\hilbert{1\mp i}{p}{\Q_p}_4.\]
		Then $a^+ a^-=\hilbert{2}{p}{\Q_p}_4=\pm i$ and
		$\frac{a^-}{a^+}=\hilbert{i}{p}{\Q_p}_4=\pm i$.
		Therefore
		\[(a^+,a^-)=(\pm i,1),(\pm i,-1),(1,\pm i),(-1,\pm i).\]
		By definition, $\hilbert{(1-\zeta_8)^2}{p}{\fP_{1}}_8=a^+$ and $\hilbert{(1-\zeta_8)^2}{p}{\fP_2}_8=\iota^{-1}_2 (a^-)$. Therefore
		\[\left( \hilbert{(1-\zeta_8)^2}{p}{\fP_{1}}_8,\hilbert{(1-\zeta_8)^2}{p}{\fP_2}_8 \right)=(\pm i,1),(\pm i,-1),(1,\mp i), (-1,\mp i).\]
		In each case, we always have  $|\rho(\Lambda_{0,2})|=|\langle \rho((1-\zeta_8)^2),\rho(\zeta_8),\rho(1+\sqrt{2}) \rangle|=32$. This proves (1).
		
		For (2), we study the map
		\[\fct{\rho_4:=\rho_{\langle \fq_{2,1}\rangle, K_{2,1}/K_{0,1}}: \Lambda_{0,1}}{\mu_4 \times \mu_4 \times \mu_4}{x}{\left(\hilbert{x}{p}{\fp_1}_4, \hilbert{x}{p}{\fp_2}_4,\hilbert{x}{p}{\fq_{0,1}}_4\right).}\]
		We always have
		\[\hilbert{i}{p}{\Q_p}_4=\hilbert{p}{i}{\Q_p}^{-1}_4
		=i^{-\frac{p-1}{4}}=\pm i. \]
		Let $\tau_1,\tau_2$ be the embeddings corresponding to $\fp_1,\fp_2$ respectively. We assume that  $\tau_1(i)=i$ and $\tau_2(i)=-i$.
		Then
		\[\hilbert{i}{p}{\fp_1}_4= \tau^{-1}_1 \hilbert{\tau_1(i)}{p}{\Q_p}_4=\tau^{-1}_2 \hilbert{\tau_2(i)}{p}{\Q_p}_4=\hilbert{i}{p}{\fp_2}_4=\pm i. \]
		Hence  $\rho_4(i)=(\pm i , \pm i,-1)$ by the product formula. So  $[E_{0,1}: E_{0,1}\cap \bfN K^\times_{2,1}]= |\rho_4(E_{0,1})|=|\langle\rho_4(i)\rangle |=4$.
		
		Now we compute $\rho_4((1+i)^2)$.
		Since
		\[\hilbert{(1-i)^2}{p}{\Q_p}_4 \hilbert{(1+i)^2}{p}{\Q_p}_4 =\hilbert{1-i}{p}{\Q_p}_2 \hilbert{1+i}{p}{\Q_p}_2=\hilbert{2}{p}{\Q_p}_2=-1, \]
		we have
		\[\hilbert{(1-i)^2}{p}{\fp_1}_4=\pm 1, \quad \hilbert{(1-i)^2}{p}{\fp_2}_4=\mp 1. \]
		Hence $\rho_4((1-i)^2)=(\pm1, \mp 1, -1)$. Therefore, $[\Lambda_{0,1} : \Lambda_{0,1}\cap \bfN K^\times_{2,1}]=|\langle \rho_4((1-i)^2),\rho_4(i)\rangle| =8$. This proves (2).

		(3)  follows from the values of the following quadratic Hilbert symbols:
		\[ \hilbert{\zeta_8}{p}{\Q_p(\zeta_8)}_2= \hilbert{-i}{p}{\Q_p}_2=-1, \quad \hilbert{1+\sqrt{2}}{p}{\Q_p(\zeta_8)}_2= \hilbert{-1}{p}{\Q_p}_2=1. \qedhere \]  \end{proof}

	\begin{proof}[Proof of Theorem~\ref{thm:main1} for $p\equiv 5\bmod 8$]
		We first prove that  $2\mid\mid h_{3,2}, 2\mid\mid h_{2,1}$ and $2\nmid h_{1,2}$.
		
		We apply Gras' formula \eqref{eq:Gras} to the case
		\[K_{3,2}/K_{0,2},\ C=\langle \cl(\fq_{3,2}) \rangle,\ D=\langle  \fq_{3,2} \rangle \]
		where $\fq_{n,m}$ is the unique prime ideal  of $K_{n,m}$ above $2$.  Then  $\Lambda_D=\Lambda_{0,2}$ as in Lemma \ref{lem:5mod8 index}.  By the above computation and Lemma \ref{nakayama},   $A_{3,2}=\langle \cl(\fq_{3,2}) \rangle (2)$.  Note that $C$ is invariant under the action of  $G:=\Gal(K_{3,2}/K_{0,2})$.   We have  $A_{3,2}=A_{3,2}^{G}$.   Chevalley's  formula \eqref{chevalley} and the above computation imply that  $|A_{3,2}|=|A_{3,2}^G|=2$.
		
		Similarly,  applying Gras' formula   to the case
		\[ K_{2,1}/K_{0,1},\ C=\langle \cl(\fq_{2,1}) \rangle, D=\langle  \fq_{2,1} \rangle \]
		shows that $A_{2,1}=\langle \cl(\fq_{2,1}) \rangle (2)$. In particular, $A_{2,1}$ is invariant under the action of $\Gal(K_{2,1}/K_{0,1})$.
		Apply Chevalley's formula  to  $K_{2,1}/K_{0,1}$, we  obtain  $|A_{2,1}|=2$.
		
		By Applying  Chevalley's formula to  the extension $K_{1,2}/K_{0,2}$ and  Lemma~\ref{nakayama},   we have $2\nmid h_{1,2}$. Hence $2\nmid h_{1,1}$ by Lemma~\ref{prop: norm_surj}.

		We have $2\mid\mid h_{n,m}$  for $n\ge 2,m\ge1$ by  Proposition~\ref{prop: square_clgp} and  $2\nmid h_{1,m}$ for $n=1,m\geq 1$  by  Proposition~\ref{prop:stable_theorem}.
		
		It remains to prove that $2\nmid h_{n,0}$. The proof consists of three steps:
		
		\noindent \textbf{Step 1}: Let $\epsilon$ be the fundamental unit of $\Q(\sqrt{p})$. We show that  $\hilbert{\epsilon}{\sqrt{p}}{\sqrt{p}}_2=-1$.
		
		Write $\epsilon=\frac{a+b\sqrt{p}}{2}, a,b\in\Z$.
		Then
		\[ \hilbert{\epsilon}{\sqrt{p}}{(\sqrt{p})}_2=
		\hilbert{a/2}{\sqrt{p}}{(\sqrt{p})}_2=\hilbert{a/2}{-p}{p}_2=\leg{a/2}{p}.\]   It is well-known $\bfN_{\Q(\sqrt{p})/\Q}(\epsilon)=(\frac{a}{2})^2-p (\frac{b}{2})^2=-1$. Since  $(\frac{a}{2})^2\equiv -1 \bmod p$ and $p\equiv 5\bmod 8$, we have   $\leg{a/2}{p}\equiv (\frac{a}{2})^{\frac{p-1}{2}}\equiv -1 \bmod p$.
		
		\noindent \textbf{Step 2}: We show that  $[E_{n,0}:E_{n,0}\cap \bfN K^\times_{n+1,0}]=4$ for each $n\geq 1$.
		
		Consider the map as in Lemma~\ref{lem: hasse_norm},
		\[\fct{\rho: E_{n,0}}{\mu_2 \times \mu_2 \times \mu_2}{x}{\left(\hilbert{x}{p^{\frac{1}{2^n}}}{\infty_n}_2, \hilbert{x}{p^{\frac{1}{2^n}}}{(p^{\frac{1}{2^n}})}_2,\hilbert{x}{p^{\frac{1}{2^n}}}{\fq_{n,0}}_2\right),}\]
		where $\infty_n$ is the real place of $K_{n,0}$ such that $\infty_n(p^{\frac{1}{2^n}})<0$.
		We know $[E_{n,0}:E_{n,0}\cap \bfN K^\times_{n,0}]=|\rho(E_{n,0})|$  and $\rho(E_{n,0})\subset (\zeta_2 \times \zeta_2 \times \zeta_2)^{\prod=1}$. In particular, $|\rho(E_{n,0})|\leq 4$.
		
		Since $-1,\epsilon \in E_{n,0}$. It is enough to prove that  $|\pair{ \rho(-1),\rho(\epsilon) }| =4$.
		By \textbf{Step 1}, we have
		\[\hilbert{\epsilon}{p^{\frac{1}{2^n}}}{(p^{\frac{1}{2^n}})}_2=\hilbert{\epsilon}{-p^{\frac{1}{2^{n-1}}}}{(p^{\frac{1}{2^n}})}_2=\cdots = \hilbert{\epsilon}{-\sqrt{p}}{(\sqrt{p})}_2=-1.\]
		Therefore, $\rho(\epsilon)=(\pm 1,-1,\mp1)$. Since $\rho(-1)=(-1,1,-1)$, we have $|\pair{ \rho(-1),\rho(\epsilon) }| =4$ and hence $|\rho(E_{n,0})|=4$.
		
		\noindent \textbf{Step 3}: We prove $2\nmid h_{n,0}$ for any $n\geq 1$.
		
		We prove it by induction on $n$. The case $n=1$ is well-known.
		Assume that  $2\nmid h_{n,0}$.
		The product of  ramification indices of $K_{n+1,0}/K_{n,0}$ is $8$.  Using the result in \textbf{Step } 2, Chevalley's  formula \eqref{chevalley} for the extension $K_{n+1,0}/K_{n,0}$
		and   Lemma~\ref{nakayama} then imply $2\nmid h_{n+1,0}$.
	\end{proof}

	\subsection{The case $p\equiv 7 \bmod 16$} \label{sec:p716}
	The main purpose of this subsection is to prove Theorem  \ref{thm:main1}(3). We first give a  brief description of the proof.
	\begin{itemize}	
		
		\item Apply Gras' formula \eqref{eq:Gras} inductively to the extension $K_{n,0}/K_{n-1,0}$ to show that  $A_{n,0}$ is generated by the unique prime above $2$.  Then apply \eqref{eq:Gras} to $K_{n,1}/K_{n,0}$ to  show that $A_{n,1}$ equals the $2$-primary part of $\pair{\text{classes of primes above $2$}}$. Next  we  apply  Chevalley's formula \eqref{chevalley} to the extensions $K_{3,1}/K_{1,1}$ and $K_{2,1}/K_{1,1}$ to deduce $A_{2,1}\cong A_{3,1}\cong \Z/{2\Z}\times \Z/{2\Z}$.   Proposition~\ref{prop:stable_theorem} then implies $A_{n,1}\cong \Z/{2\Z}\times \Z/{2\Z}$ for $n\geq 2$. Finally from this one can get $A_{n,0}\cong \Z/{2\Z}$ for $n\geq 2$.
		
		\item   Apply  \eqref{eq:Gras} inductively to $K_{1,m}/K_{0,m}$ to show that $A_{1,m}$ is a quotient of $\Z/{2^{m-1}\Z}$, then use Kida's $\lambda$-invariant formula to get $|A_{1,m}|\geq 2^{m-1}$. This leads to $A_{1,m}\cong \Z/{2^{m-1}\Z}$ for any $m\geq 1$.
	\end{itemize}

	For each $n\geq 1$, $K_{n,0}$ has two real places. Let $\infty_n$ be the real place such that $\infty_n(p^{\frac{1}{2^n}})<0$. Then $ \infty_n$ is ramified in $K_{n+1,0}/K_{n,0}$, while the other real place is unramified in $K_{n+1,0}/K_{n,0}$.
	
	The prime $p$ is totally ramified as $p\co_{n,0}=\fp_{n,0}^{2^n}$ in $K_{n,0}$, where $\fp_{n,0}=(p^{\frac{1}{2^n}})$. Since $p$ is inert in $K_{0,1}$, $\fp_{n,0}$ is inert in $K_{n,1}$. Write  $\fp_{n,0}\co_{n,1}=\fp_{n,1}$. The prime $\fp_{0,1}=(p)$ is totally ramified in $K_{\infty,1}/K_{0,1}$.
	
	Since $(x+1)^{2^n}-p$ is a $2$-Eisenstein polynomial, $2$ is totally ramified as $2\co_{n,0}=\fq_{n,0}^{2^n}$ in $K_{n,0}$. Since $2$ splits in $\Q(\sqrt{-p})/\Q$, $\fq_{n,0}$ splits as $\fq_{n,0}\co_{n,1}=\fq_{n,1}\fq_{n,1}'$ in $K_{n,1}/K_{n,0}$ for each $n\geq 1$. The primes $\fq_{1,1}$ and $\fq'_{1,1}$ are totally ramified in $K_{\infty,1}/K_{0,1}$.
	
	The prime $2$ is also totally ramified as $2\co_{0,m}=\fq_{0,m}^{2^n}$ in $K_{0,m}$, where $\fq_{0,m}=(1-\zeta_{2^{m+1}})\co_{0,m}$.  The prime $\fq_{0,m}$ splits as $\fq_{0,m}\co_{1,m}=\fq_{1,m}\fq_{1,m}'$ in $K_{1,m}$ for each $m\geq 1$.

	Since $2\nmid h_{1,0}$, $\fp_{1,0}$ is principal.   If $\pi =u+v\sqrt{p}$ is a generator of $\fp_{1,0}$, we must have  $\bfN(\pi)=u^2-pv^2=2$, since $-2$ is not a square modulo $p$.   If       $\pi$ is  a totally positive generator of $\fp_{1,0}$,   then $\frac{\pi^2}{2}=\epsilon^k$ with $k$ odd, where $\epsilon$ is the fundamental unit of $K_{1,0}$.   Replace the generator $\pi$ by $\pi \epsilon^\frac{1-k}{2}$. We may assume that $\frac{\pi^2}{2}=\epsilon$.  So
	$E_{1,0}= \pair{-1, \frac{\pi^2}{2}}$.

	\begin{lem}\label{lem: h11}
		The class number $h_{1,1}$ of $K_{1,1}=\Q(\sqrt{p},i)$ is odd and
		$E_{1,1}=\pair{\frac{\pi}{1+i}, i}$.
	\end{lem}
	\begin{proof} \label{lem: K_{1,1}}
		Apply Chevalley's formula to the extension $K_{1,1}/K_{0,1}$ and Lemma~\ref{nakayama},  one has $2\nmid h_{1,1}$.
		
		By \cite[Theorem 42, Page 195]{FT93},
		\[ [E_{1,1}:\pair{\frac{\pi^2}{2},i}]=1 \text{ or } 2.\]
		Note that $\frac{\pi}{1+i}  $ is a unit and $[\pair{\frac{\pi}{1+i}, i}:\pair{\frac{\pi^2}{2},i}]=2$,  we must have $E_{1,1}=\pair{\frac{\pi}{1+i}, i}$.
	\end{proof}

	\begin{lem}\label{lem: index7mod16}
		We have
		
		$(1)$ $\hilbert{\pi}{\sqrt{p}}{\fp_{1,0}}_2=-1$ and  $\hilbert{\pi}{\sqrt{p}}{\fq_{1,0}}_2=-1$;
		
		$(2)$ $[E_{1,0}:E_{1,0}\cap \bfN K_{2,0}^\times ]=2$;
		
		$(3)$ $[E_{1,1}:E_{1,1}\cap \bfN K_{3,1}^\times ]=4$ and $
		[E_{1,1}:E_{1,1}\cap \bfN K_{2,1}^\times ]=1$.
	\end{lem}
	\begin{proof}
		(1)
		Since $\pi=u+v\sqrt{p}$ is totally positive, we have $u>0$, $u^2-pv^2=2$ and $2\nmid uv$.
		Note that   $2$ is a square modulo $v$, so $v\equiv \pm 1 \bmod 8$.  Then $u^2\equiv 9\bmod 16$ since  $p\equiv 7\bmod 16$.  In other words, $u\equiv \pm 3\bmod 8$.  We  have
		\[ \hilbert{\pi}{\sqrt{p}}{\fp_{1,0}}_2=\hilbert{u}{\sqrt{p}}{\fp_{1,0}}_2
		=\hilbert{u}{-p}{p}_2=\leg{u}{p}=\leg{-p}{u}=\leg{2}{u}=-1. \]
		
		The fourth equality is due to the quadratic reciprocity law.
		We have $\hilbert{\pi}{\sqrt{p}}{\infty_1}_2=1$ as $\pi$ is totally positive, thus $\hilbert{\pi}{\sqrt{p}}{\fq_{1,0}}_2=-1$ by the product formula.

		(2) Since the infinite place $\infty_1$  is ramified,  $-1$ is not a norm of $K_{2,0}$. For the fundamental unit $\frac{\pi^2}{2}$, we have
		\[ \hilbert{\frac{\pi^2}{2}}{\sqrt{p}}{\fp_{1,0}}_2
		=\hilbert{2}{\sqrt{p}}{\fp_{1,0}}_2=\hilbert{2}{-p}{p}_2=1, \quad  \hilbert{\frac{\pi^2}{2}}{\sqrt{p}}{\infty_1}_2=1. \]
		By the product formula,
		\[ \hilbert{\frac{\pi^2}{2}}{\sqrt{p}}{\fq_{1,0}}_2=1. \]
		Then $\frac{\pi^2}{2}$ is a norm of $K_{2,0}$ by Hasse's norm theorem. This proves (2).

		(3) We need to study the map
		\[\fct{\rho: E_{1,1}}{\mu_4 \times \mu_4 \times \mu_4}{x}{\left(\hilbert{x}{\sqrt{p}}{\fp_{1,1}}_4, \hilbert{x}{\sqrt{p}}{\fq_{1,1}}_4,\hilbert{x}{\sqrt{p}}{\fq_{1,1}'}_4\right).}\]
		Then $\rho(E_{1,1})\subset ({\mu_4 \times \mu_4 \times \mu_4})^{\prod=1}$  and
		$[E_{1,1}:E_{1,1}\cap \bfN K^\times_{3,1}]=|\rho(E_{1,1})|$.
		
		We first compute $\rho(i)$. Since $p\equiv 7\bmod 16$ and the residue field of $\fp_{1,1}$ is $\F_{p^2}$, we have
		\[ \hilbert{i}{\sqrt{p}}{\Q_p(\sqrt{p},i)}_4
		=\hilbert{\sqrt{p}}{i}{\Q_p(\sqrt{p},i)}^{-1}_4=i^{-\frac{p^2-1}{4}}=1. \]
		Thus
		\[\hilbert{i}{\sqrt{p}}{\fp_{1,1}}_4=1. \]

		Note that the localization of $K_{1,1}$ at $\fq_{1,1}$  is $\Q_2(\sqrt{p},i)=\Q_2(i)$. Note that $\sqrt{-p}\in \Q_2$.
		Since
		\[\hilbert{i}{i}{\Q_2(i)}_4=\hilbert{i}{-1}{\Q_2(i)}_4\hilbert{i}{-i}{\Q_2(i)}_4=\hilbert{i}{-1}{\Q_2(i)}_4=\hilbert{i}{i}{\Q_2(i)}_2=1,\]
		we have
		\[\hilbert{i}{\sqrt{p}}{\Q_2(i)}_4=\hilbert{i}{\sqrt{-p}}{\Q_2(i)}_4=
		\begin{cases}
		\hilbert{i}{\sqrt{-7}}{\Q_2(i)}_4=\hilbert{i}{11}{\Q_2(i)}_4,  & \text{ if } p\equiv 7 \bmod 32;\\
		\hilbert{i}{\sqrt{-23}}{\Q_2(i)}_4=\hilbert{i}{3}{\Q_2(i)}_4,  & \text{ if } p\equiv 23 \bmod 32. \\
		\end{cases}\]
		Applying the product formula to  the quartic Hilbert symbols on  $\Q(i)$  gives
		\[\hilbert{i}{11}{\Q_2(i)}_4=\hilbert{i}{11}{\Q_{11}(i)}^{-1}_4=i^{-\frac{11^2-1}{4}}=-1,\]
		\[\hilbert{i}{3}{\Q_2(i)}_4=\hilbert{i}{3}{\Q_3(i)}^{-1}_4=i^{-\frac{3^2-1}{4}}=-1.\]
		Therefore, $\hilbert{i}{\sqrt{p}}{\Q_2(i)}_4=-1$ and we have $\rho(i)=(1,-1,-1)$.
		
		Next we compute $\rho(\frac{\pi}{1+i})$.
		By (1),  we have $\pi^{\frac{p-1}{2}}\equiv -1 \bmod \fp_{1,0}$.  Since $p\equiv 7\bmod 16$,   $\pi^\frac{p^2-1}{4}\equiv 1 \bmod \fp_{1,0}$. Hence $\hilbert{\pi}{\sqrt{p}}{\fp_{1,1}}_4=1$. Since  $(1+i)^\frac{p^2-1}{4}=(2i)^\frac{p^2-1}{8}=-2^\frac{p^2-1}{8}\equiv -1 \bmod p $, we have   $\hilbert{1+i}{\sqrt{p}}{\fp_{1,1}}_4=-1$. Thus
		\[ \hilbert{\frac{\pi}{1+i}}{\sqrt{p}}{\fp_{1,1}}_4=-1. \]
		To compute $\hilbert{\frac{\pi}{1+i}}{\sqrt{p}}{\fq_{1,1}}_4$, we first compute its square:
		\[\hilbert{\frac{\pi}{1+i}}{\sqrt{p}}{\fq_{1,1}}^2_4=\hilbert{\frac{\pi}{1+i}}
		{\sqrt{p}}{\fq_{1,1}}_2=\hilbert{\pi}{\sqrt{p}}{\fq_{1,1}}_2 \hilbert{1+i}{\sqrt{p}}{\fq_{1,1}}_2,\]
		Note that $\Q_2(\sqrt{p})=\Q_2(i)$. By part (1) of Lemma \ref{lem: index7mod16} , we have \[1=\hilbert{\pi}{\sqrt{p}}{\fq_{0,1}}_2=\hilbert{\pi}{\sqrt{p}}{\Q_2
			(\sqrt{p})}_2=\hilbert{\pi}{\sqrt{p}}{\fq_{1,1}}_2.\]
		Note that    $\sqrt{-p}\equiv \pm3 \bmod 8$.   So we have the following equality of quadratic Hilbert symbols:
		\[\hilbert{1\pm i}{\sqrt{p}}{\Q_2(i)}_2=\hilbert{1\pm i}{\sqrt{-p}}{\Q_2(i)}_2=\hilbert{2}{\sqrt{-p}}{\Q_2}_2=-1. \]
		Therefore
		\[ \hilbert{\frac{\pi}{1+i}}{\sqrt{p}}{\fq_{1,1}}^2_4=1= \hilbert{\frac{\pi}{1+i}}{\sqrt{p}}{\fq'_{1,1}}^2_4.\]
		
		By the product formula we must have  $\rho(\frac{\pi}{1+i})=(-1,\pm 1,\mp 1)$. Hence $|\rho(E_{1,1})|=4$. This implies $[E_{1,1}:E_{1,1}\cap \bfN K_{3,1}^\times ]=4$.

		To compute $[E_{1,1}:E_{1,1}\cap \bfN K_{2,1}^\times ]$,  we need to consider the following map
		\[\fct{\rho': E_{1,1}}{\mu_2 \times \mu_2 \times \mu_2}{x}{\left(\hilbert{x}{\sqrt{p}}{\fp_{1,1}}_2, \hilbert{x}{\sqrt{p}}{\fq_{1,1}}_2,\hilbert{x}{\sqrt{p}}{\fq_{1,1}'}_2\right).}\]
		Then $\rho'=\rho^2$ by Proposition \ref{prop: hil}(7). Thus $\rho'(i)=\rho(i)^2=(1,1,1)$ and $\rho'(\frac{\pi}{1+i})=\rho(\frac{\pi}{1+i})^2=(1,1,1)$. Therefore $[E_{1,1}:E_{1,1}\cap \bfN K_{2,1}^\times ]=|\rho'(E_{1,1})|=1$.
	\end{proof}

	\begin{prop}\label{prop: gen by 2}
		We have
		
		$(1)$ $A_{n,0}=\pair{\cl(\fq_{n,0})}$ for $n\geq 1$ and $A_{2,0}\cong \Z/{2\Z}$;
		
		$(2)$ $A_{n,1}=\pair{ \cl(\fq_{n,1}),\cl(\fq_{n,1}') }(2)$ for $n\geq 2$.
	\end{prop}
	\begin{proof}
		(1)	We prove this  by induction.  The case $n=1$ is well-known.  Suppose the result holds for $n$.
		We apply Gras' formula \eqref{eq:Gras} to 
		$$K_{n+1,0}/K_{n,0}, C=\pair{\cl(\fq_{n+1,0})}, D=\langle \fq_{n+1,0}\rangle.$$
		Note that    $\bfN(C)=\pair{\cl(\fq_{n,0})}=A_{n,0}$ by the assumption.   The product of ramification indices is $8$.    Consider the map
		\[\fct{\rho:=\rho_{D, K_{n+1,0}/K_{n,0}}: \Lambda_{D}}{\mu_2 \times \mu_2 \times \mu_2}{x}{\left(\hilbert{x}{p^{\frac{1}{2^n}}}{\infty_n}_2, \hilbert{x}{p^{\frac{1}{2^n}}}{\fp_{n,0}}_2,\hilbert{x}{p^{\frac{1}{2^n}}}{\fq_{n,0}}_2\right).}\]
		We have $|\rho(\Lambda_D)|=[\Lambda_{D}:\Lambda_{D}\cap \bfN K^\times_{n+1,0}]$ and $\rho(\Lambda_D)\subset (\mu_2 \times \mu_2 \times \mu_2)^{\prod=1}$, in particular, $|\rho(\Lambda_D)|\leq 4$. Notice that  $\Lambda_{D}\supset \pair{\pi,\frac{\pi^2}{2},-1}$.
		
		Since $\infty_n(p^{\frac{1}{2^n}})<0$,
		\[ \hilbert{-1}{p^{\frac{1}{2^n}}}{\infty_n}_2=-1.\]
		By the norm-compatibility of Hilbert symbols,
		\[ \hilbert{-1}{p^{\frac{1}{2^n}}}{\fp_{n,0}}_2=\hilbert{-1}{p^{\frac{1}{2^{n-1}}}}
		{\fp_{n-1,0}}_2=\cdots = \hilbert{-1}{- p}{(p)}_2=-1.\]
		Then $\rho(-1)=(-1,-1,1)$.  Since $\pi$ is totally positive,
		\[ \hilbert{\pi}{p^{\frac{1}{2^n}}}{\infty_n}_2=1.\]
		By the  norm-compatibility of Hilbert symbols and the above Lemma,
		\[ \hilbert{\pi}{p^{\frac{1}{2^n}}}{\fp_{n,0}}_2=\hilbert{\pi}{(-1)^{n-1}
			\sqrt{p}}{\fp_{1,0}}_2=-1. \]
		Hence $\rho(\pi)=(1,-1,-1)$.  Therefore $|\rho(\Lambda_D)|\geq |\langle \rho(\pi),\rho(-1)\rangle|= 4$.  This shows that $|\rho(\Lambda_D)|= 4$. Then Gras' formula  and Lemma \ref{nakayama} tell us $A_{n+1,0}=\pair{\cl(\fq_{n+1,0})}(2).$ Note that $\fq^{2^n}_{n+1,0}=\fq_{1,0}=(\pi)$, so $\pair{\cl(\fq_{n+1,0})}(2)=\pair{\cl(\fq_{n+1,0})}$.  By induction, $A_{n+1,0}=\pair{\cl(\fq_{n+1,0})}.$
		
		In particular, $A_{2,0}$ is invariant under the action of $\Gal(K_{2,0}/K_{1,0})$.  Since  $E_{1,0}=\langle -1,\frac{\pi^2}{2} \rangle$,  and  $[E_{1,0}:E_{1,0}\cap \bfN K^\times_{2,0}]=2$ by the above Lemma. Applying Chevalley's formula~\eqref{chevalley} to $K_{2,0}/K_{1,0}$  gives $A_{2,0}\cong \Z/{2\Z}$.
		
		(2) We apply Gras' formula to
		\[K_{n,1}/K_{n,0}, C=\langle \cl(\fq_{n,1}),\cl(\fq'_{n,1}) \rangle, D=\langle \fq_{n,1},\fq'_{n,1} \rangle.\]
		Then $\bfN C= \langle \cl(\fq_{n,0})\rangle=A_{n,0}$ by (1). Only the two infinite places are  ramified in $K_{n,1}/K_{n,0}$,  so   $-1$ is not a norm. This shows that the index $[\Lambda_D: \Lambda_D \cap \bfN K^\times_{n+1,0}]\geq 2$. By Gras' formula and Lemma \ref{nakayama},  $A_{n,1}=\langle \cl(\fq_{n,1}),\cl(\fq'_{n,1}) \rangle(2)$.
	\end{proof}
	
	\begin{thm}\label{thm: 7mod16}
		For $p\equiv 7\mod{16}$, we have $A_{n,1}\cong \Z/{2\Z}\times\Z/{2\Z}$ and $A_{n,0}\cong \Z/{2\Z}$ for any $n\geq 2$.
	\end{thm}
	\begin{proof}
		The extension $K_{\infty,1}/K_{1,1}$ satisfies  $\RamHyp$ and  $\Gal(K_{n+2,1}/K_{n,1})$ is cyclic of order $4$ for each $n\geq 1$. By Proposition \ref{prop:stable_theorem}, to show  $A_{n,1}\cong \Z/{2\Z}\times\Z/{2\Z}$, it suffices  to show   $A_{2,1}\cong A_{3,1}\cong \Z/{2\Z}\times \Z/{2\Z}$.
		
		Let $G_{2,1}=\Gal(K_{2,1}/K_{1,1})$.
		By Proposition~\ref{prop:  gen by 2}, $A_{2,1}=\pair{ \cl(\fq_{2,1}),\cl(\fq_{2,1}') }(2)=A_{2,1}^{G_{2,1}}$. Since $h_{1,1}$ is odd,  $\cl(\fq_{2,1})^2=\cl(\fq_{1,1}\co_{2,1})$ has odd order. In other words, $A_{2,1}$ is a quotient of $\Z/{2\Z}\times \Z/{2\Z}$. Note that $A_{2,1}=A_{2,1}^{G_{2,1}}$.  The product of ramification indices of $K_{2,1}/K_{1,1}$ is $8$. By Lemma \ref{lem: index7mod16} and   Chevalley's formula \eqref{chevalley} for $K_{2,1}/K_{1,1}$, we obtain    $|A_{2,1}|=|A_{2,1}^{G_{2,1}}|=4$.  So $A_{2,1}\cong \Z/{2\Z}\times \Z/{2\Z}$.
		
		By Proposition \ref{prop: gen by 2},  $A_{3,1}=A_{3,1}^{G_{3,1}}$ where $G_{3,1}=\Gal(K_{3,1}/K_{1,1})$. The product of ramification indices of $K_{3,1}/K_{1,1}$ is $64$.  By Lemma \ref{lem: index7mod16} and Chevalley's formula for   $K_{3,1}/K_{1,1}$,  we get $|A_{3,1}|=|A_{3,1}^{G_{3,1}}|=4$.  Since the norm map from $A_{3,1}$ to $A_{2,1}$ is surjective by Lemma \ref{prop: norm_surj}, we must have $A_{3,1} \cong \Z/{2\Z}\times \Z/{2\Z}$.

		Now we compute $A_{n,0}$. Since $K_{n,1}/K_{n,0}$ is ramified at the real places, the norm map from $A_{n,1}$ to $A_{n,0}$ is surjective by Lemma \ref{prop: norm_surj}. In particular, $A_{n,0}$ is a quotient of $\Z/{2\Z}\times \Z/{2\Z}$. We know that $A_{n,0}$ is cyclic by Proposition~\ref{prop: gen by 2}.
		Since the norm map from $A_{n,0}$ to $A_{2,0}\cong \Z/{2\Z}$ is  surjective, we must have $A_{n,0}\cong \Z/{2\Z}$ for $n\geq 2$.
	\end{proof}
	
	To compute the $2$-class group of  $K_{1,m}$ for $m\geq 1$, we first note that $K_{1,m}$ is  the $m$-th layer of the cyclotomic $\Z_2$-extension of $K_{1,1}$.
	
	\begin{prop}
		For $p\equiv 7\mod{16}$, we have $A_{1,m}=\pair{\cl(\fq_{1,m})}(2)$ for $m\ge 1$.
	\end{prop}
	\begin{proof}
		We  first reduce the result to the case $m=2$.   Suppose $A_{1,2}=\pair{\cl(\fq_{1,2})}(2)$.  Note that $K_{1,\infty}/K_{1,1}$ is totally ramified at $\fq_{1,1}$ and  $\fq'_{1,1}$, and unramified outside $\fq_{1,1}$ and $\fq'_{1,1}$.   Applying Gras' formula \eqref{eq:Gras} to		\[K_{1,2}/K_{1,1},\ C_1=\pair{\cl(\fq_{1,2})},\ D_1=\langle \fq_{1,2}\rangle\]
		gives
		\[[\Lambda_{D_1}:\Lambda_{D_1}\cap \bfN K^\times_{1,2}]=2.\]
		Next we apply Gras' formula to
		\[K_{1,3}/K_{1,2},\ C_2=\pair{\cl(\fq_{1,3})},\ D_2=\langle \fq_{1,3}\rangle. \]
		Note that  $\bfN(C)(2)=A_{1,2}$. To prove $A_{1,3}=C_2$, we need to  prove that $[\Lambda_{D_2}:\Lambda_{D_2}\cap \bfN K^\times_{1,3}]=2$ by Lemma \ref{nakayama}.  Note that  $K_{1,2}=K_{1,1}(\sqrt{-i})$ and $K_{1,3}=K_{1,2}(\sqrt{\zeta_8})$.   We need to study the following two maps:
		\[\fct{\rho_1=\rho_{D_1, K_{1,2}/K_{1,1}}: \Lambda_{D_1}}{\mu_2 \times \mu_2 }{x}{\left(\hilbert{x}{-i}{\fq_{1,1}}_2, \hilbert{x}{-i}{\fq'_{1,1}}_2\right)}\]
		and
		\[\fct{\rho_2=\rho_{D_2, K_{1,3}/K_{1,2}}: \Lambda_{D_2}}{\mu_2 \times \mu_2 }{x}{\left(\hilbert{x}{\zeta_8}{\fq_{1,2}}_2, \hilbert{x}{\zeta_8}{\fq'_{1,2}}_2\right).}\]
		
		We have  $|\rho_2(\Lambda_2)|=[\Lambda_{D_2}:\Lambda_{D_2}\cap \bfN K^\times_{1,3}]\leq 2$ by Lemma~\ref{lem: hasse_norm}.
		Note that $\Lambda_{D_1}\subset \Lambda_{D_2}$. By the norm-compatible property of Hilbert symbols, $\hilbert{x}{\zeta_8}{\fq_{1,2}}_2=\hilbert{x}{-i}{\fq_{1,1}}_2$.  So  the following  diagram is commutative:		
		\[\xymatrix{
			\Lambda_{D_2} \ar[r]^-{\rho_{2}}&  \mu_2\times \mu_2\\
			\Lambda_{D_1} \ar[ru]^-{\rho_{1}} \ar@{^(->}[u] &
		}\]
		Thus  $2=|\rho_1(\Lambda_{D_1})| \leq |\rho_2(\Lambda_{D_2})|\leq 2 $ and $[\Lambda_{D_2}:\Lambda_{D_2}\cap \bfN K^\times_{1,3}]=2$, which implies that  $A_{1,3}=\pair{\cl(\fq_{1,3})}(2)$ by Lemma~\ref{nakayama}.  Repeating this argument, we get $A_{1,m}=\pair{\cl(\fq_{1,m})}(2)$ for $m\geq 2$.
		
		Consider the case
		\[ K/F=K_{1,2}/K_{0,2},\ C=\langle \cl(\fq_{1,2})\rangle,\ D=\langle \fq_{1,2}\rangle.\]
		Note that $C$ is a $\Gal(K_{1,2}/K_{0,2})$-submodule of $A_{1,2}$, since for $\sigma \in \Gal(K_{1,2}/K_{0,2})$, $\sigma(\fq_{1,2})\fq_{1,2}=\fq_{0,2}\co_{1,2}=(1-\zeta_8)\co_{1,2}$, in other words, $\sigma(\cl(\fq_{1,2}))=\cl(\fq_{1,2})^{-1}$. If we can show $[\Lambda_D:\Lambda_D\cap \bfN K^\times_{1,2}]=2$, then by Gras' formula \eqref{eq:Gras} and Lemma~\ref{nakayama}, we have  $A_{1,2}=\pair{\cl(\fq_{1,2})}(2)$.
		
		Note that  $\Lambda_D=\langle 1-\zeta_8,\zeta_8,1+\sqrt{2}\rangle$ and the ramified places in $K_{1,2}/K_{0,2}$ are $\fp_{0,2}$ and $\fp'_{0,2}$,  where $\fp_{0,2}\fp'_{0,2}=p\co_{0,2}$. By Lemma~\ref{lem: hasse_norm}, for the map
		\[\fct{\rho= \rho_{D, K_{1,2}/K_{0,2}}:\Lambda_D}{\mu_2\times \mu_2}{x}{\left(\hilbert{x}{p}
			{\fp_{0,2}}_2,\hilbert{x}{p}{\fp'_{0,2}}_2\right),}\]
		we have $|\rho(\Lambda_D)|=[\Lambda_D:\Lambda_D\cap \bfN K^\times_{1,2}]\leq 2$.  To show $|\rho(\Lambda_D)|=2$, it suffices to show that $\rho$ is not trivial.
		Let us compute $\rho(1-\zeta_8)$. For $p\equiv 7 \bmod 16$, the conjugate of $\zeta_8$ over $\Q_p$ is $\zeta_8^{-1}$.
		By the norm-compatible property of Hilbert symbols, we have
		\[ \hilbert{1-\zeta_8}{p}{\fp_{0,2}}_2=
		\hilbert{1-\zeta_8}{p}{\Q_p(\zeta_8)}_2 =\hilbert{(1-\zeta_8)(1-\zeta_8^{-1})}{p}{\Q_p}_2
		=\hilbert{2+\zeta_8+\zeta^{-1}_8}{p}{\Q_p}_2.
		\]
		By Hensel's Lemma, we have
		\[\hilbert{2+\zeta_8+\zeta^{-1}_8}{p}{\Q_p}_2=1 \Leftrightarrow 2+\zeta_8+\zeta^{-1}_8 \bmod p \text{ is a square } \Leftrightarrow 2+\zeta_8+\zeta^{-1}_8\in (\Q^\times_p)^2 .\]
		Notice that $(\zeta_{16}+\zeta_{16}^{-1})^2=2+\zeta_8+\zeta^{-1}_8$.
		Since $p\equiv 7\bmod{16}$, $\Frob_p(\zeta_{16}+\zeta_{16}^{-1})=\zeta_{16}^7+\zeta_{16}^{-7}=-(\zeta_{16}+\zeta_{16}^{-1})$, where $\Frob_p$ is the Frobenius element of $\Gal(\overline{\Q}_p/\Q_p)$.
		Thus $\zeta_{16}+\zeta_{16}^{-1}\notin\Q_p$ and we have $\hilbert{1-\zeta_8}{p}{\fp_{0,2}}_2=-1$.
	\end{proof}
	
	\begin{thm}\label{thm: 7mod16 cyc} For $p\equiv 7\mod{16}$ and $m\geq 1$, $A_{1,m}\cong \Z/2^{m-1}\Z$.
	\end{thm}
	\begin{proof}
		Note that $A_{1,1}$ is trivial and $\fq_{1,m}^{2^{m-1}}=\fq_{1,1}$.  We have $A_{1,m}=\langle \cl(\fq_{1,m})\rangle(2)$ is a quotient of $\Z/2^{m-1}\Z$. Since  $h_{1,m}\mid h_{1,m+1}$ by Lemma \ref{prop: norm_surj},  if $|A_{1,m}|<2^{m-1}$ for some $m$,  we must have  $|A_{1,k}|=|A_{1,k+1}|$ for some $k$.  Then  $|A_{1,n}|=|A_{1,k}|$ for any $n \ge k$ by Proposition \ref{prop:stable_theorem}.  But Kida's  formula \cite[Theorem 1]{Kid79}  shows that the $\lambda$-invariant of the cyclotomic $\Z_2$-extension of $\Q(\sqrt{-p})$ is $1$.
		In particular,
		the $2$-class numbers of $\Q(\sqrt{-p},\zeta_{2^{m+1}}+\zeta^{-1}_{2^{m+1}})$ are unbounded when $m\rightarrow \infty$. Thus the $2$-class numbers of $\Q(\sqrt{-p},\zeta_{2^{m+1}})=K_{1,m}$ are also unbounded by Lemma~\ref{prop: norm_surj}. We get a contradiction.
	\end{proof}

	\begin{proof}[Proof of Theorem \ref{thm:main1}(3)]
		Theorem \ref{thm:main1}(3) is just the combination of Theorem \ref{thm: 7mod16} and Theorem \ref{thm: 7mod16 cyc}.
	\end{proof}

	\subsection{Congruence property of the relative fundamental unit} We are now ready to prove Theorem \ref{thm: units}. 
	We assume $p\equiv 7\bmod{16}$ and use the same notations as in \S~\ref{sec:p716}. 

	To prove this theorem, we need an explicit reciprocity law for a real quadratic field $F$. We view $F\subset \R$ by fixing an embedding.
	For a prime ideal $\fp$ with odd norm and $\gamma\in \co_F$ prime to $\fp$, define the Legendre symbol  $\qleg{\gamma}{\fp} \in \{\pm 1\} $   by the congruence $\qleg{\gamma}{\fp}\equiv \gamma^{\frac{\bfN\fp-1}{2}} \bmod \fp$.
	For coprime $\gamma,\delta\in\co_F$ with $(2,\delta)=1$, define  $\qleg{\gamma}{\delta}:=\prod_{\fp\mid \delta}\qleg{\gamma}{\fp	}^{v_\fp(\delta)}$.  So by definition $\qleg{\gamma}{\delta}=1$ if $\delta \in \co_F^\times$. 
	
	For $\gamma,\delta\in\co_F \setminus\{0\}$, define
	\[\{\gamma,\delta\}=(-1)^{\frac{\sgn(\gamma)-1}{2}\cdot\frac{\sgn(\delta)-1}{2}}\]
	where $\sgn(x)=1$ if $x>0$ and $\sgn(x)=-1$ if $x<0$.
	Note that $ \{\gamma,\delta_1\}\{\gamma,\delta_2\}=\{\gamma,\delta_1\delta_2\}.$
	\begin{thm}\label{thm:quad_leg_rec}
		Assume that $\gamma_1,\delta_1,\gamma_2,\delta_2\in \co_F$ have odd norms, $\gamma_1$ and $\delta_1$ are coprime, $\gamma_2$ and $\delta_2$ are coprime, and $\gamma_1\equiv \gamma_2,\delta_1\equiv \delta_2\bmod 4$. Then
		\[\qleg{\gamma_1}{\delta_1}\qleg{\delta_1}{\gamma_1}\qleg{\gamma_2}{\delta_2}\qleg{\delta_2}{\gamma_2}=\{\gamma_1,\delta_1\}\{\gamma_1',\delta_1'\}\{\gamma_2,\delta_2\}\{\gamma_2',\delta_2'\}.\]
		where $\xi'$ is the conjugate of $\xi\in F$.
	\end{thm}
	\begin{proof}
		This follows from \cite[Lemma~12.12, Lemma~12.13, Lemma~12.16]{Lem05} directly.
	\end{proof}
	
	\begin{proof}[Proof of Theorem~\ref{thm: units}]
		(1)    	Note that $E_{2,0}/E_{1,0}$ is an abelian group of rank $1$. We claim that $E_{2,0}/E_{1,0}$ is torsion-free. Otherwise, there exists $u\in E_{2,0}\setminus E_{1,0}$ such that $u^j\in E_{1,0}$ for some $j\geq 2$.  Then $K_{2,0}=K_{1,0}(u)$.   The norm of $u$ respect to the extension $K_{2,0}/K_{1,0}$ is $u \zeta u=\zeta u^2\in E_{1,0}$ for some $\zeta\in \zeta_j \cap K_{2,0}$.   So $\zeta=\pm 1$. Thus $u^2\in E_{1,0}$ and this implies that  $K_{2,0}/K_{1,0}$ is unramified at $p$. This contradicts to the fact that $K_{2,0}/K_{1,0}$ is ramified at $p$. This proves  the claim.

		Let $\eta\in E_{2,0}$ such that its image in $E_{2,0}/E_{1,0}$ is a generator of $E_{2,0}/E_{1,0}$. Then clearly $E_{2,0}=\langle \eta, \epsilon,-1 \rangle$.	By Lemma \ref{lem: index7mod16}, $\epsilon \in \bfN K^\times_{2,0}$.  Let $G=\Gal(K_{2,0}/K_{1,0})$.
		Since $A^G_{2,0}=\langle \fq_{2,0} \rangle$ and $\fq_{2,0}$ is a $G$-invariant fractional ideal,  by \cite[Proposition 1.3.4]{Gre},  $E_{1,0}\cap \bfN K^\times_{2,0}=\bfN E_{2,0}$ and in particular $\epsilon \in \bfN E_{2,0}$. Therefore we must have $\bfN(\pm\eta\epsilon^k)=\epsilon$. Replacing $\eta$ by $\sgn(\eta)\eta\epsilon^k$, then $\eta$ is totally positive since $\epsilon$ is totally positive, $\bfN(\eta)=\epsilon$ and $E_{2,0}=\langle \eta, \epsilon,-1 \rangle$.
		
		(2) We first reduce it to the case $\eta'=\eta$. Suppose the result holds for $\eta$. For any  $\eta'\in E_{2,0}$ such that $\bfN(\eta')=\epsilon$, we can write $\eta'=\sgn(\eta') \eta^k \epsilon^{s}$ with $k=1-2s$.   Firstly, one easily see that $\epsilon \equiv \pm 1 \bmod \sqrt{p}$.  We claim that $\epsilon \equiv 1 \bmod \sqrt{p}$.  Since $\epsilon=\bfN(\eta)=\eta \overline{\eta}$, we have $\epsilon \equiv \eta \overline{\eta} \equiv \eta^2  \bmod \sqrt[4]{p}$. Therefore, $\epsilon$ is a square in $\co_{2,0}/(\sqrt[4]{p})\cong \F_p$.  Because $-1$ is not a square in $\F_p$, we obtain $\epsilon \equiv 1 \bmod \sqrt{p}$. Then $\eta'\equiv \sgn(\eta') (-1)^k \equiv  -\sgn(\eta')  \bmod \sqrt[4]{p}$. Write $\eta=\alpha+\beta \sqrt[4]{p}$ with $\alpha,\beta \in \Z[\sqrt{p}]$. By the assumption we have $\fq \mid\mid \alpha$ and $\fq \nmid \beta$. It is easy to check that for odd $k$,  $\fq \mid\mid \alpha_k$ also where $\eta^k=\alpha_k+\beta_k\sqrt[4]{p}$ with $\alpha_k,\beta_k \in \Z[\sqrt{p}]$. Thus we have  $v_\fq (\Tr(\eta'))=v_\fq(2\epsilon^s \alpha_k)=v_\fq(2\epsilon^s \alpha)=3$.

		From now on we prove the result holds for $\eta=\alpha+\beta\sqrt[4]{p}$. Write $\alpha=a+b\sqrt{p}$ and $\beta=c+d\sqrt{p}$ with $a,b,c,d\in \Z$.
		Since the infinite place is ramified in $K_{2,0}$, we have $\bfN_{K_{2,0}/\Q}(\eta)=1$. Hence $\bfN_{K_{2,0}/\Q}(\eta)\equiv a^4 \equiv  1 \bmod \sqrt[4]{p}$. Since $p\equiv 7\bmod 16$, we have $\eta \equiv a\equiv \pm 1 \bmod \sqrt[4]{p}$.
		
		Let $G=\Gal(K_{3,0}/K_{2,0})$. By Proposition \ref{prop: gen by 2} and Theorem \ref{thm: 7mod16} tell us  $|A_{3,0}|=|A_{3,0}^{G}|=|A_{2,0}|=2$.  Applying Chevalley's formula \eqref{chevalley}  on $K_{3,0}/K_{2,0}$ gives $[E_{2,0}:\bfN K^\times_{3,0}\cap E_{2,0}]=4$. This implies
		$\left(\hilbert{\eta}{\sqrt[4]{p}}{\infty_2}, \hilbert{\eta}{\sqrt[4]{p}}{(\sqrt[4]{p})},\hilbert{\eta}{\sqrt[4]{p}}{\fq_{2,0}}\right)\neq (1,1,1)$. Therefore $\hilbert{\eta}{\sqrt[4]{p}}{(\sqrt[4]{p})}=\hilbert{\eta}{\sqrt[4]{p}}{\fq_{2,0}}=-1$ by  the  totally positivity of $\eta$  and the product formula. Hence $\eta $ is not a square modulo $\sqrt[4]{p}$ and we must have $\eta \equiv -1 \bmod \sqrt[4]{p}$.

		Write $\alpha=\pi^t\alpha_0$  with $\pi \nmid \alpha_0$, recall that $\pi$ is the totally positive generator of $\fq$ such that $\epsilon=\frac{\pi^2}{2}$.   Now $t=v_\fq(\Tr(\frac{\eta}{2}))= v_\fq(\Tr(\eta))-2$, so our goal is to prove  $t=1$.  Note that $\alpha$ and $\alpha_0$ are positive.
		Write $\epsilon=x+y\sqrt{p}$, $\pi=u+v\sqrt{p}$.
		By Lemma~\ref{lem: index7mod16}, $u$ and $v$ are both odd and $v\equiv \pm 1\bmod 8$. From $\epsilon=\frac{\pi^2}{2}$ and $\bfN(\pi)=u^2-pv^2=2$, we obtain $8\parallel x=u^2-1=pv^2+1$ and $y\equiv \pm 3\bmod 8$.

		If   $y\equiv 3\bmod8$, then  $\epsilon\equiv -\sqrt{p}\bmod 4$.
		Take $(\alpha_0,-\sqrt{p},\alpha_0,\epsilon)$ in Theorem~\ref{thm:quad_leg_rec}, since $\alpha_0>0, \sqrt{p}\epsilon'>0$, we have
		\[\qleg{\alpha_0}{-\sqrt{p}}\qleg{-\sqrt{p}}{\alpha_0}\qleg{\alpha_0}
		{\epsilon}\qleg{\epsilon}{\alpha_0}
		=\{\alpha_0,-\sqrt{p}\epsilon\}\{\alpha'_0,\sqrt{p}\epsilon'\}
		=1.\]
		Since $\alpha^2-\sqrt{p}\beta^2=\epsilon$, we have 
		\[ \qleg{\alpha^2-\sqrt{p}\beta^2}{\alpha_0}= \qleg{-\sqrt{p}}{\alpha_0}=\qleg{\epsilon}{\alpha_0}.\]
		
		By definition, $\qleg{\alpha_0}{\epsilon}=1$. Combine the above two equalities,  $\qleg{\alpha_0}{-\sqrt{p}}=1$. By  Lemma \ref{lem: index7mod16}, $\qleg{\pi}{-\sqrt{p}}=\hilbert{\pi}{\sqrt{p}}{\sqrt{p}}_2=-1$.
		Thus we have \[-1=\qleg{\alpha}{-\sqrt{p}}=\qleg{\pi}{-\sqrt{p}}^t\qleg{\alpha_0}{
			-\sqrt{p}}=(-1)^t,\]
		which means that  $t$ is odd in this case.

		If  $y\equiv -3 \bmod 8$, then $\epsilon^{-1}=x-y\sqrt{p}$ with  $-y\equiv 3\bmod 8$ and  $\bfN({\eta^{-1}})=\epsilon^{-1}$.
		Repeating the above argument, we obtain  $v_\fq(\Tr(\frac{\eta^{-1}}{2})$ is odd.  Let $\bar{\eta}=\alpha-\beta \sqrt[4]{p}$.  We have $\Tr({\eta^{-1}})=\Tr(\bar{\eta}{\epsilon^{-1}})={\epsilon^{-1}} \Tr(\overline{\eta})={\epsilon^{-1}} \Tr(\eta)$.  Therefore $t=v_\fq(\frac{\Tr(\eta)}{2})=v_\fq(\frac{\Tr(\eta^{-1})}{2})+v_\fq(\epsilon^{-1})=v_\fq(\frac{\Tr(\eta^{-1})}{2})$ is also odd.
		
		Finally let us  prove  $t=1$.  
		Recall that  $\eta=a+b\sqrt{p}+(c+d\sqrt{p})\sqrt[4]{p}$ with $a,b,c,d\in \Z$.  Since $t$ is odd, we have $\pi \mid a+b\sqrt{p}$ and $\pi\nmid c+d\sqrt{p}$.  Then $c\not \equiv d \bmod 2$. From $\bfN(\eta)=\epsilon=x+y\sqrt{p}$ we have
		$a^2+pb^2-2cdp=x$.
		Assume   $t\ge 3$,  i.e. $2\pi \mid a+b\sqrt{p}$.  We must have $2\parallel a$ and $2\parallel b$ or $4\mid a$ and $4\mid b$. In both cases, $x\equiv -2cd p \bmod 8$. 
		Since $8\mid x$, we have $4\mid cd$. But exactly one of $c$ and $d$ is odd, $y=2ab-c^2-pd^2\equiv d^2-c^2\equiv \pm 1\bmod 8$, which is a contradiction to $y\equiv \pm 3\bmod 8$. Thus $t=1$.
	\end{proof}


	%
	%

	
	\bibliographystyle{alpha}

\end{document}